\documentclass[a4paper,11pt]{article}
\usepackage[affil-it]{authblk}
\usepackage[english]{babel}
\usepackage{blindtext}
\usepackage{abstract}
\usepackage{graphicx, color}
\usepackage{mathptmx}      
\usepackage{amsmath, amsthm, amssymb}
\usepackage{enumerate}
\usepackage{multirow}
\usepackage{url}

\newcommand{\IE}{\textit{i.e., }}
\newcommand{\CF}{\textit{cf.\ }} 
\newcommand{\EG}{\textit{e.g., }}

\newcommand{\ZZ}{{\mathbb Z}}
\newcommand{\QQ}{{\mathbb Q}}
\newcommand{\RR}{{\mathbb R}}
\newcommand{\CC}{{\mathbb C}}
\newcommand{\tr}[1]{{\hspace{0.5mm}{}^t\hspace{-0.5mm}#1}}
\newcommand{\abs}[1]{ {\left\lvert #1 \right\rvert} }

\newtheorem{theorem}{Theorem}
\newtheorem{lemma}{Lemma}
\newtheorem{proposition}{Proposition}
\newtheorem{definition}{Definition}
\newtheorem{example}{Example}
\newtheorem{remark}{Remark}

\title{Packing theory derived from phyllotaxis and products of linear forms}
\author[1]{S. E. Graiff Zurita\thanks{s-graiff@math.kyushu-u.ac.jp}}
\author[2]{B. Kane\thanks{bkane@hku.hk}}
\author[3]{R. Oishi-Tomiyasu\thanks{tomiyasu@imi.kyushu-u.ac.jp (corresponding author)}}

\affil[1]{%
    Graduate School of Mathematics, Kyushu University}
\affil[2]{%
    Department of Mathematics, University of Hong Kong}
\affil[3]{%
      Institute of Mathematics for Industry (IMI), Kyushu University}

\date{}

\begin{document}

\maketitle

\begin{abstract}
\textit{Parastichies} are spiral patterns observed in plants and numerical patterns generated using golden angle method. 
We generalize this method by using Markoff theory 
and the theory of product of linear forms, to obtain a theory for packing of 
Riemannian manifolds of general dimensions $n$ with a locally diagonalizable metric, including the Euclidean spaces.
For example, packings in a plane with logarithmic spirals and in a 3D ball (3D analogue of the Vogel spiral)
are newly obtained.
Using this method, we prove that it is possible to generate almost uniformly distributed point sets on any real analytic Riemannian surfaces in a local sense.
We also discuss how to extend the packing to the whole manifold in some special cases including the Vogel spiral.
The packing density is bounded below by approximately 0.7 for surfaces and 0.38 for 3-manifolds under the most general assumption.

\end{abstract}


%
%
\section{Introduction}
\label{intro}

In botany, spiral patterns observed in plants, such as sunflower heads, cacti, and pinecones, are called \textit{parastichies}.
Parastichies can also be observed in patterns that are numerically generated on a surface with circular symmetry using golden angle method.
In the golden angle method, a new point is generated for each $2 \pi \varphi \approx 137.5^\circ$-rotation, 
where $\varphi = 1/(1+ \gamma_1)$ is the golden angle defined using the golden ratio $\gamma_1 = (1+\sqrt{5})/2$.
With this method, packing patterns 
have been generated on a cylindrical surface \cite{Bravais1837}, a disk known as the Vogel spiral \cite{Mathai74}, \cite{Vogel79} (Figure~\ref{fig: Vogel spiral}), surfaces of revolution \cite{Ridley86}, sphere surface for mesh generation on globe \cite{Swinbank2006}, \cite{Hardin2016}, and Poincar{\' e} disc \cite{Sadoc2013}. 


\begin{figure}[htbp]
\scalebox{0.4}{\includegraphics{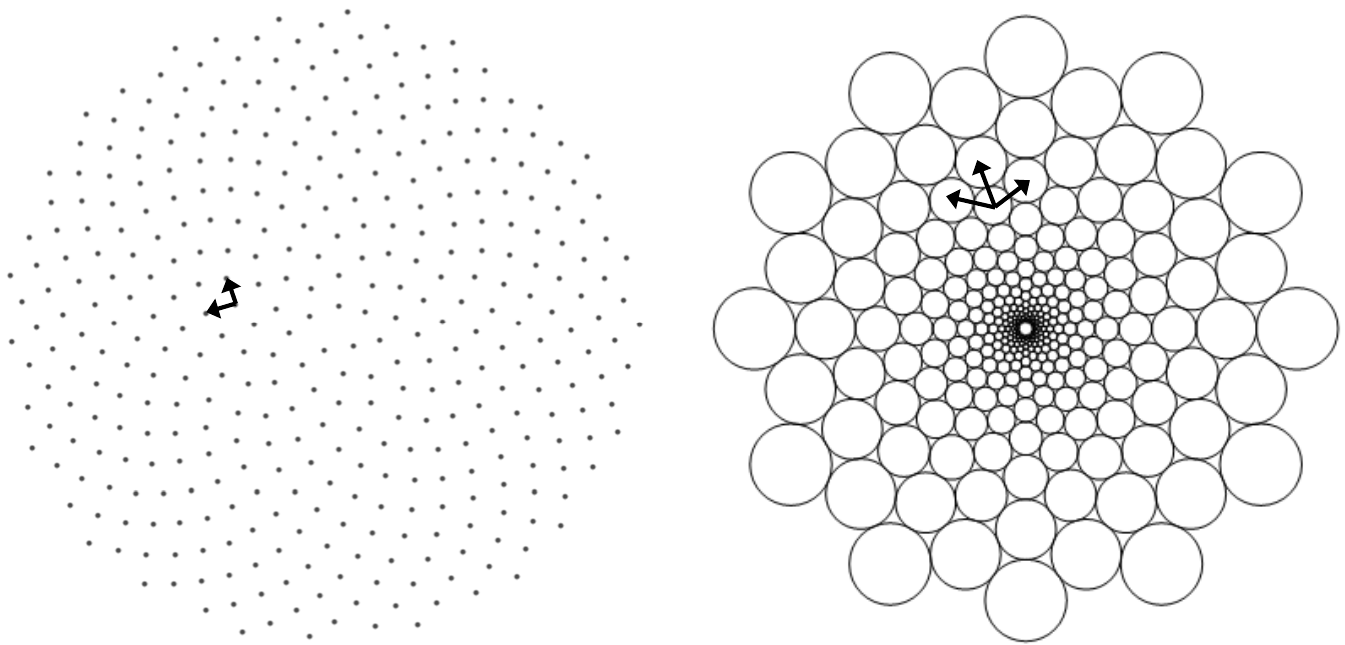}}
\caption{{\bf Left}: Vogel spiral $\sqrt{n} e^{2 \pi i n /(1+\gamma_1)}$ ($n > 0$: integer) and images of the lattice shortest vectors (arrows) that indicate the directions of parastichies, {\bf Right}: Doyle spiral of type $(12, 24)$ dealt with in Example \ref{exa: Doyle spiral}.
}
\label{fig: Vogel spiral} 
\end{figure}

The generated patterns can be regarded as the image of the lattice with the basis:
\begin{eqnarray}\label{eq: basis for the Vogel spiral}
\begin{pmatrix}
1 \\
0
\end{pmatrix}, 
\begin{pmatrix}
\varphi \\
\epsilon_0
\end{pmatrix}, \quad \varphi = \frac{1}{1+ \gamma_1},\ \epsilon_0: \text{constant.}
\end{eqnarray}
Variations of the Vogel spiral can be created by substituting different $\varphi$~\cite{Naylor2002}.
While the golden angle method is applicable to surfaces with circular symmetry, 
it has not been known how to generalize it on general surfaces or manifolds in higher dimensions, maintaining the packing density in a certain range. The golden angle method for higher dimensions has been posed as an open problem \cite{Lodkin2019}, \cite{Akiyama2020}.
The main goal of this paper is to establish a general theory and a concrete method for that.

As a \textit{packing}, we shall consider an arrangement of non-overlapping $n$-dimensional balls of the identical radius. 
Although the case of varying radii 
is not the scope of this article, 
the Doyle spiral  (Figure~\ref{fig: Vogel spiral}) has also been addressed as a pattern of parastichies~\cite{Arnold11} and known as the case of conformal mapping. 
In fact, the Doyle spiral and the Vogel spiral can be constructed within the same framework (Example~\ref{exa: Doyle spiral}).


For the above goal, it is necessary to determine the optimal lattice of general rank that will be used in our packing method.
As shown in (2) of Theorem~\ref{thm: theorem 1}, even 
the lattice of rank 2 given by (\ref{eq: basis for the Vogel spiral}) is not optimal for bounding the packing density by a larger lower limit. 

For any full-rank lattice $L \subset \RR^n$, let $\Delta(L)$ be the packing density (\ref{eq: definition of den}) of the lattice packing given by $L$. 
For any $B_n \in GL_n(\RR)$, let $L(B_n)$ be the lattice generated by the column vectors of $B_n$. 
The following problem will be discussed in Section~\ref{sec: Most density-stable lattices in R^2 and R^3} for 
determining the optimal lattices, using known theorems for \textit{products of linear forms.} 
\begin{description}
\item[{\bf Problem}:]
Determine the lattice basis $B \in GL_n(\RR)$ with $\Delta_n^\prime = \Delta^\prime(L(B))$.
\begin{eqnarray}
\Delta_n^\prime &:=& \sup_{ B \in GL_n(\RR), } \Delta^\prime(L(B)), \\
\Delta^\prime(L(B)) &:=& \inf_{ D \in GL_n(\RR)\text{: diagonal} } \Delta( L(D B) ). \nonumber
\end{eqnarray}
\end{description}


%
The answer for the above problem 
enables us to expand the scope of the golden angle method to higher dimensions.
The packing densities of the obtained aperiodic packings 
are bounded below approximately by 
$\Delta_n^\prime$ for each dimension $n$: 
\begin{eqnarray*}
\Delta_2^\prime = \frac{\pi}{ 2 \sqrt{5} } \approx 0.702, & &
\Delta_3^\prime = \frac{ \sqrt{3} \pi }{ 14 } \approx 0.389, \\ 
\Delta_4^\prime \ge \frac{ \pi^2 }{ 10\sqrt{29} } \approx 0.183, & & 
\Delta_5^\prime \ge \frac{ 5 \sqrt{5} \pi^2 }{ 12 \cdot 11^{2} } \approx 0.076. 
\end{eqnarray*}



In addition to such a lattice, 
a map $f({\mathbf x}) = (f_1, \ldots, f_N)$ from an open subset ${\mathfrak D} \subset \RR^n$ to $\RR^N$ ($n \le N$) with the 
Jacobian matrix $D f({\mathbf x}) = (\partial f_i / \partial x_j)_{1 \le i, j \le n}$ satisfying ($\star$) and ($\star\star$) below, is used in our packing method. 
\begin{itemize}
\item[($\star$)] $\tr{(Df)} D f$ is an invertible diagonal matrix for any ${\mathbf x} \in {\mathfrak D}$ (diagonalizable). 

\item[($\star\star$)] $\det \tr{(Df)} Df = c^2$ for some constant $c \ne 0$ (volume-preserving).
\end{itemize}
For a fixed pair of a map $f: {\mathfrak D} \rightarrow \RR^N$ and a lattice $L \subset \RR^n$, 
a packing of $f({\mathfrak D})$ is provided as $f(L \cap {\mathfrak D})$.

As a result of ($\star\star$), the mesh generated as the Voronoi diagram of the point packing is \textit{equiareal}, \IE its Voronoi cells have approximately the same volume.
($\star$) means that the metric on $f({\mathfrak D}) \subset \RR^N$ induced from the Euclidean metric of $\RR^N$ is diagonal.
It is known that diagonalization of the metric is possible for any $C^\infty$ Riemannian manifolds of dimensions $n \le 3$ in a local sense. 
The case of $n = 2$ follows from the existence of conformal metric $\lambda(x, y) (dx^2 + d y^2)$, and the case of $n = 3$ was proved in \cite{DeTurck84}. 

The main purpose of Section~\ref{sec: Application to aperiodic packings}
is to present new aperiodic packings obtained by our generalization, discussing the system of partial differential equations (PDEs) that provide $f$ satisfying ($\star$), ($\star\star$) for $n = N$.
The case of only ($\star$) and $n=N$ is known as the problem on the $n$-orthogonal curvilinear coordinate systems and has been studied extensively\cite{Darboux10}, \cite{Zakharov98}, but the case with additional constraint ($\star\star$) has been  much less studied with a few exceptions such as \cite{Kuusela86}. 

The results for the general case $n \le N$ are summarized in Section~\ref{sec: Application to aperiodic packings of the Riemannian manifolds}.
In Theorem~\ref{thm: theorem 3}, we prove that any real analytic Riemannian surface has an atlas $\{ (U_\alpha, \varphi_\alpha) \}_{\alpha \in I}$ such that every $U_\alpha$ has the metric our packing method can be applied to. 
The same is clearly true for $n = 1$ in case of piecewise differentiable curves, because ($\star$) and ($\star\star$) means that $f$ is a parametrization of $f({\mathfrak D})$ by the arc-length.

Subsequently,
a family of the PDE solutions for general dimensions is provided (Theorem~\ref{thm: theorem 4}) to explain 
the self-similarity observed among the aperiodic packings provided in Section~\ref{sec: Application to aperiodic packings}.
Another family obtained from solutions of the inviscid Burgers equation $u_t + u u_x = 0$,
also has the same self-similarity (Example~\ref{exa: Case of inviscid Burgers equation}).


As a motivation for Theorem~\ref{thm: theorem 4}, we remark that
the Vogel spiral and the phyllotaxis model can be regarded as a growing disk and a growing cylindrical surface, respectively.
This aspect of the golden angle method seems to have been not clear, while it is only applicable to surfaces with circular symmetry.
Even after the generalization, the obtained packings still have the following commonly features and an appearance reminiscent of biological growth as seen in Figures~\ref{fig: i-c solutions},~\ref{fig: 3D Vogel spiral surface}.

Firstly, the packings presented in this article 
are the images of lattice points contained in a rectangle ${\mathfrak D}$ (or a rectangular parallelepiped for the 3D case), 
which suggests that there is an inductive way to construct the PDE solutions.
Secondly, $f({\mathfrak D})$ has a codimension-1 foliation $\{ L_t \}$ such that any leaves $L_t \subset f({\mathfrak D})$ are similar to each other in $\RR^N$.
More specifically, if the last entry $x_n$ is separated so that ${\mathbf x} = ({\mathbf x}_{n-1}, x_n)$ and ${\mathfrak D} = {\mathfrak D}_2 \times I$ (${\mathfrak D}_2 \subset \RR^{n-1}$, $I \subset \RR$), 
every $f$ can be represented by 
\begin{eqnarray}\label{eq: f with self similarity}
f({\mathbf x}) = e^{\alpha(h({\mathbf x}))} U(h({\mathbf x})) f_2({\mathbf x}_{n-1}) + {\mathbf v}_0
\end{eqnarray}
for some functions $h: {\mathfrak D} \rightarrow \RR$, $\alpha: h({\mathfrak D}) \rightarrow \RR$, $U: h({\mathfrak D}) \rightarrow O(N)$, $f_2: {\mathfrak D}_2 \rightarrow \RR^N$ and ${\mathbf v}_0 \in \RR^N$.
Thus, if $t = h({\mathbf x})$ is considered as the time variable, 
the image of $\{ {\mathbf x} \in {\mathfrak D} : t = h({\mathbf x}) \}$ by $f$ has the identical shape as $f_2({\mathfrak D}_2)$ for any $t \in I$.

Theorem~\ref{thm: theorem 4} provides a method to construct $(n+1)$-dimensional Riemannian manifolds with a diagonal and constant-determinant metric from manifolds of dimension $n$ with such a metric.
In particular, for dimensions $n = 2, 3$, Theorem~\ref{thm: theorem 4} can provide a large family of 
lattice maps in the form of Eq.(\ref{eq: f with self similarity}) that fulfill ($\star$) and ($\star\star$).


As a result, the golden angle method has been generalized by using volume-preserving maps, based on algebraic properties of certain special lattices.
Aperiodic packing has a number of applications.
When quasicrystals were discovered \cite{Shechtman1984},
the Penrose tiling was immediately suggested as a mathematical model of aperiodic structures 
with a long-range order \cite{Mackay82}.
More recently, iterative algorithms for generating circle packings on Riemann surfaces have been studied \cite{Collins2003}, inspired by the Koebe-Andreev-Thurston theorem \cite{Koebe1936}, \cite{Andreev1970}, \cite{Thurston1978} and the Thurston conjecture proved in \cite{Rodin87}.
The generalized golden angle method will also broaden the range of applications. 

\section*{Methods to color point sets}

All the presented packings are 2D or 3D scatter plots displayed by Wolfram Mathematica.
The points are colored by either of two methods. The first one colors each point according to the local packing density in its neighborhood.
In the considered situation in which the lattice basis $B$ and the map $f$ on ${\mathfrak D} \subset \RR^n$ are specified,
the local density around $f({\mathbf x})$ can be approximated 
from the packing density of the lattice $L((D f ({\mathbf x})) B)$ as justified in Theorem~\ref{thm: theorem 4}.

The second one colors each point, according to its \textit{birth time},
by considering some $x_i$ among ${\mathbf x} = (x_1, \ldots, x_n)$ as the time when the point $f({\mathbf x})$ is generated
(\EG Figure~\ref{fig: i-c solutions}). 
This method allows easy observation of the self-similarity hidden in the image.

\section*{Notation and symbols}



The continued fractions are represented using squared brackets $[ \ ]$ as follows.
\begin{eqnarray*}
[a_0, a_1, a_2, \cdots, a_n, \cdots ] := a_0 + \frac{1}{a_1+} \frac{1}{a_2+} \frac{1}{a_3+} \cdots \frac{1}{a_{n}+} \cdots.
\end{eqnarray*}
The \textit{$n$'th convergent} $p_n / q_n$ of $\varphi = [a_0, a_1, a_2, \cdots]$ is the ratio of 
the coprime integers $p_n$, $q_n$ that satisfy
$p_n / q_n = [a_0, a_1, a_2, \cdots, a_n]$.
If there is an integer $m > 0$ such that $a_{m + n} = a_n$ for any $n \ge 0$, 
$[a_0, a_1, \ldots ]$ is called \textit{purely periodic}. 
In such a case,
$[a_0, a_1, \ldots ]$ is abbreviated as $[\overline{a_0, a_1, \ldots, a_{m-1}}]$.

For any lattice $L \subset \RR^n$ of full rank, 
we denote the \textit{packing density} by 
\begin{eqnarray}
\Delta(L) &:=& 
\frac{ {\rm vol}(B(1)) (\min\, L)^{n/2} }{ 2^n {\rm vol}(\RR^n / L) }
= \frac{ \pi^{n/2} }{ \Gamma(n/2+1) } \frac{ (\min\, L)^{n/2} }{ 2^n {\rm vol}(\RR^n / L) }, \label{eq: definition of den}
\end{eqnarray}
where $\min\, L$ is defined by 
\begin{eqnarray}\label{eq: definition of min}
\min\, L := \{ \abs{ l }^2 : 0 \ne l \in L \},
\end{eqnarray}
and ${\rm vol}(B(1))$ and ${\rm vol}(\RR^n / L)$ are 
the volumes of the $n$-ball with radius 1 and $\RR^n / L$.

$B \in GL_n(\RR)$ is called a \textit{basis matrix} of $L$
if the columns of $B$ are a basis of $L$.
Such a lattice $L$ is also denoted as $L(B)$ explicitly.
For any diagonal $D \in GL_n(\RR)$, the lattice $L(D B)$ is also denoted by $D \cdot L$.

We define the Lagrange number and Markoff spectrum for Theorem~\ref{thm: theorem 1}. 
\begin{definition}\label{dfn: Lagrange number}
For any real number $\alpha$, the supremum $\sup M$ of $M$ that fulfills (*) is called \textit{the Lagrange number of 
$\alpha$}, and denoted by ${\mathcal L}(\alpha)$.
\begin{description}
\item[(*)] Infinitely many rationals $p/q$ satisfy $\abs{ \alpha - p/q } < \displaystyle\frac{1}{M q^2}$.
\end{description}
The set $\{ {\mathcal L}(\alpha) : \alpha \in \RR \setminus \QQ \}$ is called \textit{the Lagrange spectrum}.
\end{definition}

For any irrational $\alpha$, its Lagrange number can be calculated from the continued fraction expansion
$\alpha = [a_0, a_1, a_2, \ldots ]$ by the following formula (\CF Proposition~1.22, \cite{Aigner2013}):
$$
{\mathcal L}(\alpha) = \limsup_{n \rightarrow \infty} \left( [a_{n+1}, a_{n+2}, \ldots ] + [0, a_{n}, a_{n-1}, \ldots, a_1 ] \right).
$$

For any indefinite binary quadratic form $f(x, y) = a x^2 + b x y + c y^2$ with real coefficients, its \textit{discriminant} $d(f)$ and $m(f)$ are defined by:
\begin{eqnarray*}
d(f) &=& b^2 - 4ac, \\
m(f) &=& \inf \left\{ \abs{f(x, y)} : 0 \ne (x, y) \in \ZZ^2 \right\}.
\end{eqnarray*}

\begin{definition}\label{dfn: Markoff spectrum}
The set of ${\mathcal M}(f) := \sqrt{d(f)} / m(f)$ of all the indefinite binary quadratic forms over $\RR$ is \textit{the Markoff spectrum}.
\end{definition}

The Markoff theorem states that the Lagrange spectrum and the Markoff spectrum coincide below 3 \cite{Markoff1880}.
In fact, any indefinite binary quadratic forms $f$ over $\RR$ with ${\mathcal M}(f) < 3$,
has a root $\alpha$ of $f(x, 1) = 0$ that fulfills ${\mathcal M}(f) = {\mathcal L}(\alpha)$.
Such an $\alpha$ with ${\mathcal L}(\alpha) < 3$
is a quadratic irrational with the continued fraction expansion 
$\alpha = [a_0, \ldots, a_n, \gamma]$, where $\gamma$ is equal to one of the following $\gamma_m$:
\begin{eqnarray}\label{eq: Markoff quadratic irrationals}
\gamma_m = \frac{m + 2u + \sqrt{9m^2 - 4}}{2m},
\end{eqnarray}
where $m$ is one of the Markoff numbers $m = 1, 2, 5, 13, \ldots$ \IE positive integers that fulfill the Markoff equation for some 
integers $m_1, m_2$:
$$
m^2 + m_1^2 + m_2^2 = 3 m m_1 m_2.
$$
The integer $0 \le u \le m/2$ is the solution of $u^2 \equiv -1$ mod $m$.

%
%
\section{Parastichies from a viewpoint of lattice-basis reduction and Markoff theory}
\label{sec: Most density-stable lattices in R^2 and R^3}


Mathematically, parastichies are the images of lines that connect lattice points with the shortest vectors.
In order to see this more precisely, 
it is useful to review Markov theory from a viewpoint of the reduction theory of lattices.
Proposition~\ref{prop: Selling reduced basis} describes a result on indefinite binary quadratic forms
in terms of the reduction of positive-definite quadratic forms.

Let $\varphi_1 > 0 > \varphi_2$ be real numbers with the continued fraction expansions: 
\begin{eqnarray*}
\varphi_1 &=& [a_0, a_1, a_2, \cdots, a_n, \cdots ], \\
-\varphi_2^{-1} &=& [a_{-1}, a_{-2}, \cdots, a_{-n}, \cdots ]. 
\end{eqnarray*}
The \textit{doubly infinite sequence} $\{ a_{n} \}_{n=-\infty}^\infty$ associated with $(\varphi_1, \varphi_2)$ is obtained from the expansions.
For simplicity, $\varphi_1, \varphi_2$ are assumed to be irrational. Hence $a_n > 0$ for any $n \ne 0, -1$.

Let $p_n^{(+)} / q_n^{(+)} = [a_0, a_1, a_2, \cdots, a_n]$, $p_n^{(-)} / q_n^{(-)} = [a_{-1}, a_{-2}, a_{-3}, \cdots, a_{-n-1}]$ be 
the $n$'th convergent of $\varphi_1$ and $-\varphi_2^{-1}$. 
For $n = -1, -2$, we put:
$$
\begin{pmatrix}
p_{-1}^{(+)} & p_{-2}^{(+)} \\
q_{-1}^{(+)} & q_{-2}^{(+)} \\
\end{pmatrix}
=
\begin{pmatrix}
p_{-1}^{(-)} & p_{-2}^{(-)} \\
q_{-1}^{(-)} & q_{-2}^{(-)} \\
\end{pmatrix}
= \begin{pmatrix}
1 & 0 \\
0 & 1
\end{pmatrix}.
$$

If $\varphi_1 = [\overline{a_0, a_1, \ldots, a_{m-1}}]$, the following equality holds for the conjugate $\overline{\varphi_1}$ of $\varphi_1$ (Lemma~1.28, \cite{Aigner2013}):
$$
-1/\overline{\varphi_1} = [\overline{a_{m-1}, a_{m-2}, \ldots a_0}].
$$
Therefore, if $\varphi_2 = \overline{\varphi_1}$, $\{ a_{n} \}_{n=-\infty}^\infty$ is also a periodic sequence.

For any $\epsilon > 0$, let $L_{\varphi_1, \varphi_2, \epsilon}$ be the lattice with the following basis:
$$
{\mathbf b}_{0} :=
\begin{pmatrix}
\epsilon^{-1/2} \\
\epsilon^{1/2}
\end{pmatrix}, \quad 
{\mathbf b}_{-1} :=
\begin{pmatrix}
-\epsilon^{-1/2} \varphi_1 \\
-\epsilon^{1/2} \varphi_2 
\end{pmatrix}.
$$

A new basis ${\mathbf b}_{n}, {\mathbf b}_{n-1}$ of $L_{\varphi_1, \varphi_2, \epsilon}$ can be provided by:
\begin{eqnarray*}
{\mathbf b}_{n} &=& 
\begin{cases}
p_{n-1}^{(+)} {\mathbf b}_{0} + q_{n-1}^{(+)} {\mathbf b}_{-1} & \text{if } n > 0, \\
(-1)^n (q_{-n-2}^{(-)} {\mathbf b}_{0} - p_{-n-2}^{(-)} {\mathbf b}_{-1}) & \text{if } n < -1. 
\end{cases}
\end{eqnarray*}

In fact, from the definition, 
\begin{eqnarray*}
\begin{pmatrix}
{\mathbf b}_{n} & {\mathbf b}_{n-1}
\end{pmatrix}
&=&
\begin{cases}
\begin{pmatrix}
{\mathbf b}_{0} & {\mathbf b}_{-1}
\end{pmatrix}
\begin{pmatrix}
p_{n-1}^{(+)} & p_{n-2}^{(+)} \\
q_{n-1}^{(+)} & q_{n-2}^{(+)} \\
\end{pmatrix} & \text{if } n \ge 0, \\
(-1)^n
\begin{pmatrix}
{\mathbf b}_{0} & {\mathbf b}_{-1} 
\end{pmatrix}
\begin{pmatrix}
q_{-n-2}^{(-)} & -q_{-n-1}^{(-)} \\
-p_{-n-2}^{(-)} & p_{-n-1}^{(-)} \\
\end{pmatrix} & \text{if } n \le 0, 
\end{cases} \\
&=&
\begin{cases}
\begin{pmatrix}
\epsilon^{-1/2} & 0 \\
0 & \epsilon^{1/2}
\end{pmatrix}
\begin{pmatrix}
1 & -\varphi_1 \\
1 & -\varphi_2
\end{pmatrix}
\begin{pmatrix}
a_{0} & 1 \\
1 & 0 \\
\end{pmatrix}
\cdots
\begin{pmatrix}
a_{n-1} & 1 \\
1 & 0 \\
\end{pmatrix} & \text{if } n \ge 0, \\ 
\begin{pmatrix}
\epsilon^{-1/2} & 0 \\
0 & \epsilon^{1/2}
\end{pmatrix}
\begin{pmatrix} 
1 & -\varphi_1 \\
1 & -\varphi_2 \\
\end{pmatrix}
\begin{pmatrix}
0 & 1 \\
1 & -a_{-1} \\
\end{pmatrix}
\cdots
\begin{pmatrix}
0 & 1 \\
1 & -a_{n} \\
\end{pmatrix} & \text{if } n \le 0.
\end{cases}
\end{eqnarray*}

Hence, 
\begin{eqnarray*}
\begin{pmatrix}
{\mathbf b}_{n} & {\mathbf b}_{n-1}
\end{pmatrix}
&=&
\begin{cases}
\begin{pmatrix}
{\mathbf b}_{n-1} & {\mathbf b}_{n-2}
\end{pmatrix}
\begin{pmatrix}
a_{n-1} & 1 \\
1 & 0 \\
\end{pmatrix} & \text{if } n \ge 1, \\ 
\begin{pmatrix}
{\mathbf b}_{n+1} & {\mathbf b}_{n}
\end{pmatrix}
\begin{pmatrix}
0 & 1 \\
1 & -a_{n} \\
\end{pmatrix} & \text{if } n \le -1.
\end{cases}
\end{eqnarray*}

Thus, 
${\mathbf b}_{n} = a_{n-1} {\mathbf b}_{n-1} + {\mathbf b}_{n-2}$ holds for any $n \in \ZZ$.
From $a_0, a_{-1} \ge 0$, $a_n > 0$ ($n \ne 0, -1$)
and the following, $\{ {\mathbf b}_n \cdot {\mathbf b}_{n-1} \}_{n=-\infty}^\infty$ increases monotonically:
$$
{\mathbf b}_n \cdot {\mathbf b}_{n-1} - {\mathbf b}_{n-1} \cdot {\mathbf b}_{n-2}
= 
a_{n-1} {\mathbf b}_{n-1} \cdot {\mathbf b}_{n-1}.
$$

In addition, 
$${\mathbf b}_n \cdot {\mathbf b}_{n-1}
=
\begin{cases}
\epsilon^{-1} q_{n-1}^{(+)} q_{n-2}^{(+)} \left( p_{n-1}^{(+)} / q_{n-1}^{(+)} - \varphi_1 \right)\left( p_{n-2}^{(+)} / q_{n-2}^{(+)} - \varphi_1 \right) \\
\hspace{5mm} + \epsilon q_{n-1}^{(+)} q_{n-2}^{(+)} \left( p_{n-1}^{(+)} / q_{n-1}^{(+)} - \varphi_2 \right)\left( p_{n-2}^{(+)} / q_{n-2}^{(+)} - \varphi_2 \right) & n \ge 0, \\
-\epsilon \varphi_2^2 q_{-n-1}^{(-)} q_{-n-2}^{(-)} \left( p_{-n-1}^{(-)} / q_{-n-1}^{(-)} + \varphi_2^{-1} \right)\left( p_{-n-2}^{(-)} / q_{-n-2}^{(-)} + \varphi_2^{-1} \right) \\
\hspace{5mm} - \epsilon^{-1} \varphi_1^2 q_{-n-1}^{(-)} q_{-n-2}^{(-)} \left( p_{-n-1}^{(-)} / q_{-n-1}^{(-)} + \varphi_1^{-1} \right)\left( p_{-n-2}^{(-)} / q_{-n-2}^{(-)} + \varphi_1^{-1} \right) & n \le 0.
\end{cases}
$$
In both cases, the first terms in the right-hand side converge to $0$ as $n \rightarrow \pm \infty$. 
The second terms diverge owing to $q_{n}^{(+)} \rightarrow \infty$ and $q_{-n}^{(-)} \rightarrow \infty$ ($n \rightarrow \infty$).
Hence, $\displaystyle\lim_{n \rightarrow \infty} {\mathbf b}_n \cdot {\mathbf b}_{n-1} = \infty$.
Similarly, $\displaystyle\lim_{n \rightarrow -\infty} {\mathbf b}_n \cdot {\mathbf b}_{n-1} = -\infty$
is obtained.

The term \textit{superbase} was first used in \cite{Conway97} to explain the Selling reduction \cite{Selling1874}.
\begin{definition}
For any lattice $L$ of rank 2, 
a basis $u_1, u_2 \in L$ and $u_3 = -u_1 - u_2$ are \textit{a Selling-reduced superbase} if $u_1, u_2, u_3$ satisfy
$u_1 \cdot u_2 \le 0$, $u_1 \cdot u_3 \le 0$, and $u_2 \cdot u_3 \le 0$.
\end{definition}

\begin{proposition}\label{prop: Selling reduced basis}
For any $\varphi_1 > 0 > \varphi_2$ and $\epsilon > 0$, suppose that a lattice $L_{\varphi_1, \varphi_2, \epsilon}$ and its basis vectors ${\mathbf b}_n$ ($n \in \ZZ$) are defined as above.
Let $N$ be the integer that fulfills
$$
n \ge N \Leftrightarrow {\mathbf b}_n \cdot {\mathbf b}_{n-1} \ge 0.
$$
Let $1 \le d \le a_{N-1}$ be the smallest integer that satisfies 
$$
	(d {\mathbf b}_{N-1} + {\mathbf b}_{N-2}) \cdot {\mathbf b}_{N-1} \ge 0.
$$ 
In this case, the following $u_1, u_2, u_3$ are a Selling reduced superbase of $L_{\varphi_1, \varphi_2, \epsilon}$.
$$u_1 := {\mathbf b}_{N-1}, \quad u_2 := (d-1) {\mathbf b}_{N-1} + {\mathbf b}_{N-2}, \quad u_3 := -u_1-u_2 = -(d {\mathbf b}_{N-1} + {\mathbf b}_{N-2}).$$
In addition, one of the following holds:
\begin{enumerate}[(i)]
\item $\abs{u_1} < \abs{u_2}, \abs{u_3}$.
\item $\abs{u_2} \le \abs{u_1} < \abs{u_3}$. In this case, $d = 1$, $u_2 = {\mathbf b}_{N-2}$.
\item $\abs{u_3} \le \abs{u_1} < \abs{u_2}$. In this case, $d = a_{N-1}$, $u_3 = -{\mathbf b}_{N}$.
\item $\abs{u_2}, \abs{u_3} \le \abs{u_1}$. In this case, $d = a_{N-1} = 1$ and $u_2 = {\mathbf b}_{N-2}$, $u_3 = -{\mathbf b}_{N}$.
\end{enumerate}
In particular, one of ${\mathbf b}_{N-2}, {\mathbf b}_{N-1}$, or ${\mathbf b}_{N}$ is the shortest vector of $L_{\varphi_1, \varphi_2, \epsilon}$.

\end{proposition}

\begin{proof}
From the choice of $N$, the following hold:
$$
	{\mathbf b}_{N} \cdot {\mathbf b}_{N-1} = (a_N {\mathbf b}_{N-1} + {\mathbf b}_{N-2}) \cdot {\mathbf b}_{N-1} \ge 0, \quad {\mathbf b}_{N-1} \cdot {\mathbf b}_{N-2} < 0.
$$
Since $(x {\mathbf b}_{N-1} + {\mathbf b}_{N-2}) \cdot {\mathbf b}_{N-1}$ 
monotonically increases as a function of $x$, there is an integer $1 \le d \le a_{N-1}$ as stated above.

We shall show that $u_1, u_2, u_3$ are a Selling-reduced superbase; clearly, $u_1, u_2$ are a basis of $L_{\varphi_1, \varphi_2, \epsilon}$.
From the definition, $u_1 \cdot u_2 < 0$ and $u_1 \cdot u_3 \le 0$.
Hence, we only need to show that $u_2 \cdot u_3 \le 0$.
For the proof, $N \ge 1$ may be assumed; in fact, for any integer $-M < N$,
$\begin{pmatrix} {\mathbf b}_0 & {\mathbf b}_{-1} \end{pmatrix}$ 
can be replaced with $\begin{pmatrix} {\mathbf b}_{-M} & {\mathbf b}_{-M-1} \end{pmatrix}$, 
by replacing $\varphi_1, \varphi_2$ with $[a_{-M}, a_{-M+1}, \ldots, ]$, $-[0, a_{-M-1}, a_{-M-2}, \ldots ]$,
and changing $\epsilon$ accordingly. 
As a result, $N$ can be set to a positive integer.

From the process of the continued fraction expansion, 
$a_N$ ($N \ge 1$) is the largest integer among all the $c$'s for which the following two have different signatures:
\begin{eqnarray*}
\frac{ p_{N-1}^{(+)} }{ q_{N-1}^{(+)} } - \varphi_1 
&=& [a_0, a_1, \ldots, a_{N-1}] - \varphi_1, \\
\frac{ c p_{N-1}^{(+)} + p_{N-2}^{(+)} }{ c q_{N-1}^{(+)} + q_{N-2}^{(+)} } - \varphi_1
&=& [a_0, a_1, \ldots, a_{N-1}, c] - \varphi_1.
\end{eqnarray*}
In particular, the following two have the same signature for any $1 \le r \le a_N$:
\begin{eqnarray*}
\frac{ (r-1) p_{N-1}^{(+)} + p_{N-2}^{(+)} }{(r-1) q_{N-1}^{(+)} + q_{N-2}^{(+)}} - \varphi_1
&=& [a_0, a_1, \ldots, a_{N-1}, r-1] - \varphi_1, \\
\frac{ r p_{N-1}^{(+)} + p_{N-2}^{(+)} }{ r q_{N-1}^{(+)} + q_{N-2}^{(+)} } - \varphi_1
&=& [a_0, a_1, \ldots, a_{N-1}, r] - \varphi_1.
\end{eqnarray*}

Thus, $u_2 \cdot u_3 \le 0$ is obtained from the following: 
\begin{small}
\begin{eqnarray*}
- u_2 \cdot u_3
&=& 
((d-1) q_{N-1}^{(+)} + q_{N-2}^{(+)})(d q_{N-1}^{(+)} + q_{N-2}^{(+)}) \\
& & \hspace{-20mm} \times 
\begin{Bmatrix} 
\epsilon^{-1} \left( \frac{ (d-1) p_{N-1}^{(+)} + p_{N-2}^{(+)} }{(d-1) q_{N-1}^{(+)} + q_{N-2}^{(+)}} - \varphi_1 \right)
\left( \frac{ d p_{N-1}^{(+)} + p_{N-2}^{(+)} }{ d q_{N-1}^{(+)} + q_{N-2}^{(+)} } - \varphi_1 \right)
+ \epsilon \left( \frac{ (d-1) p_{N-1}^{(+)} + p_{N-2}^{(+)} }{(d-1) q_{N-1}^{(+)} + q_{N-2}^{(+)} } - \varphi_2 \right)
\left( \frac{ d p_{N-1}^{(+)} + p_{N-2}^{(+)} }{ d q_{N-1}^{(+)} + q_{N-2}^{(+)} } - \varphi_2 \right)
\end{Bmatrix}.
\end{eqnarray*}
\end{small}

Therefore, $u_1, u_2, u_3$ is Selling reduced, which implies that 
one of $u_1, u_2, u_3$ is the shortest vector of $L$.

As for the second statement, the following equation is obtained in a similar way as the above:
\begin{eqnarray}
\abs{u_2} \le \abs{u_1}
& \Leftrightarrow &
-(u_2 - u_1) \cdot u_3 \le 0 \nonumber \\
& \Rightarrow &
\left( \frac{ (d-2) p_{N-1}^{(+)} + p_{N-2}^{(+)} }{(d-2) q_{N-1}^{(+)} + q_{N-2}^{(+)}} - \varphi_1 \right)
\left( \frac{ d p_{N-1}^{(+)} + p_{N-2}^{(+)} }{ d q_{N-1}^{(+)} + q_{N-2}^{(+)}} - \varphi_1 \right) < 0,
\label{eq: cmp1} \\
\abs{u_3} \le \abs{u_1}
& \Leftrightarrow &
-u_2 \cdot (u_3 - u_1) \le 0 \nonumber \\
& \Rightarrow &
\left( \frac{ (d+1) p_{N-1}^{(+)} + p_{N-2}^{(+)} }{ (d+1) q_{N-1}^{(+)} + q_{N-2}^{(+)}} - \varphi_1 \right)
\left( \frac{ (d-1) p_{N-1}^{(+)} + p_{N-2}^{(+)} }{ (d-1) q_{N-1}^{(+)} + q_{N-2}^{(+)}} - \varphi_1 \right) < 0. \quad \quad \label{eq: cmp2}
\end{eqnarray}

Because the values in the parentheses of Eq.(\ref{eq: cmp1}) have different signatures owing to the inequality, 
$\abs{u_2} \le \abs{u_1}$ implies $d = 1$.
Similarly, $\abs{u_3} \le \abs{u_1}$ implies $d = a_{N-1}$. \qed

\end{proof}

The determination of $\Delta_n^\prime$, $\Delta^\prime(L(B_n))$
is equivalent to that of the $\lambda_n^\prime, \lambda^\prime(B_n)$.
\begin{eqnarray*}
\lambda_n^\prime &:=& \sup_{B_n \in GL_n(\RR), \det B_n = \pm 1} \lambda^\prime(B_n), \\
\lambda^\prime(B_n) &:=& \inf_{ {D \in GL_n(\RR)\text{: diagonal}}\atop{\det D = \pm 1} } 
\min\, L(D B_n). \nonumber
\end{eqnarray*}
In fact, for any $B_n \in GL_n(\RR)$ with the determinant $\pm 1$,
\begin{eqnarray}
\Delta^\prime(L(B_n)) &=& \frac{ ( \pi \lambda^\prime(B_n) / 4)^{n/2} }{ \Gamma(n/2 + 1) }. \label{eq: delta^prime(B_n) and lambda^prime(B_n) }
\end{eqnarray}

The determination of $\lambda^\prime_n$ and $\lambda^\prime(B_n)$
is boiled down to the following problem about the products of linear forms $\lambda_n$ and $\lambda(B_n)$, as proved in Lamma~\ref{lem: lemma 1}.
\begin{eqnarray}
\lambda_n &:=& \sup_{ {B_n = (b_{ij}) \in GL_n(\RR),}\atop{\det B = \pm 1}} \lambda(B_n) \label{eq: definition of lambda_n}, \\
\lambda(B_n) &:=& \inf_{0 \ne (x_1, \ldots, x_n) \in \ZZ^n } \abs{ \prod_{i=1}^n (b_{i1} x_1 + \cdots + b_{in} x_n) }. \nonumber
\end{eqnarray}

\begin{lemma}\label{lem: lemma 1}
\begin{enumerate}[(1)]
\item For any $B = (b_{ij}) \in GL_n(\RR)$ with $\lambda(B_n) \ne 0$, 
the following equality holds:
\begin{eqnarray}\label{eq: inequality by AM-GM}
\lambda^\prime(B_n) = n \left( \lambda(B_n) \right)^{2/n}.
\end{eqnarray}

\item 
If a totally real algebraic number field $K$ of degree $n$ has discriminant $d_K$, then, $\lambda_n \ge d_K^{-1/2}$. 
Furthermore, 
for each $n$, some $B_n$ attains the supremum $\Delta_n^\prime$, and
\begin{eqnarray*}
\Delta^\prime_n &=& 
\frac{ (\pi \lambda_n^\prime/4)^{n/2} }{ \Gamma(n/2+1) }
= \frac{ (\pi n/4)^{n/2} \lambda_n }{ \Gamma(n/2+1) } \ge \frac{ (\pi n/4)^{n/2} d_K^{-1/2} }{ \Gamma(n/2+1) }.
\end{eqnarray*}
\end{enumerate}
\end{lemma}

\begin{proof}
\begin{enumerate}[(1)]
\item From the inequality of arithmetic and geometric means,
the part $\ge$ is proved since we have the following: 
\begin{eqnarray*}
\lambda^\prime(B_n) &=&
\inf_{ { d_1, \ldots, d_n \in \RR,\ d_1 \cdots d_n = \pm 1, }\atop{ 0 \ne (x_1, \ldots, x_n) \in \ZZ^n } } \sum_{i=1}^n d_i^2 (b_{i1} x_1 + \cdots + b_{in} x_n)^2, \\
\lambda(B_n) &=&
\inf_{0 \ne (x_1, \ldots, x_n) \in \ZZ^n} \abs{ \prod_{i=1}^n (b_{i1} x_1 + \cdots + b_{in} x_n) }.
\end{eqnarray*}

Furthermore, $ > $ is impossible, because
if $\abs{ \prod_{i=1}^n (b_{i1} x_1 + \cdots + b_{in} x_n) } = v$
for some $0 \ne (x_1, \ldots, x_n) \in \ZZ^n$, 
the left-hand side of Eq.(\ref{eq: inequality by AM-GM}) cannot be more than $n v^{2/n}$,
which is seen by putting $d_i = \abs{ v^{1/n} / (b_{i1} x_1 + \cdots + b_{in} x_n) }$.

\item
Let $\sigma_1, \ldots, \sigma_n$ be distinct embeddings of $K$ into $\CC$ over $\QQ$, $b_1, \ldots, b_n$ be a basis of the ring ${\mathfrak o}_K$ of integers of $K_n$ over $\ZZ$, and $B_n \in GL_n(\RR)$ be the matrix with $b_j^{\sigma_i}$ in the $(i, j)$-entry.
From $\abs{ \det B_n } = \sqrt{ d_K }$ and $\displaystyle\min_{0 \ne \alpha \in {\mathfrak o}_K} \abs{ \sigma_1(\alpha) \cdots \sigma_n(\alpha) } = 1$, we have
$$
\lambda_n \ge \lambda(d_K^{-1/2n} B_n) = d_K^{-1/2}.
$$
This implies that $\lambda_n > 0$, and thus, the following star body is finite type. 
$$
\{ (x_1, \ldots, x_n) \in \RR^n : \abs{ x_1 \cdots x_n } \le 1 \}.
$$
Hence, $\lambda(B_n) = \lambda_n$ holds for some $B_n \in GL_n(\RR)$ with $\det B_n = \pm 1$ (Theorem 9 of $\S$17, chap.3 in \cite{Gruber87}). 
For such a $B_n$, 
$\lambda^\prime(B_n) = \lambda_n^\prime$ and
$\Delta^\prime(L(B_n)) = \Delta_n^\prime$
hold owing to Eqs.(\ref{eq: delta^prime(B_n) and lambda^prime(B_n) }) and (\ref{eq: inequality by AM-GM}).
The last inequality is also clear.
\qed
\end{enumerate}
\end{proof}


The relation between Markoff theory and phyllotaxis was not clarified until very recently \cite{Bergeron2019}.
Theorem 1 of \cite{Bergeron2019} handles a special case of our Eq.(\ref{eq: 1st eq of thm}),
as their growth capacity is a constant multiple of the packing density. 
Markoff theory can give us more insights. As seen in (2) of Theorem \ref{thm: theorem 1}, the golden lattice given by $B_2$ is more optimal than the lattice by Eq.(\ref{eq: basis for the Vogel spiral}) that has been used for the golden angle method.

\begin{theorem}\label{thm: theorem 1}
\begin{enumerate}[(1)]

\item\label{item: optimum case for n=2,3} 
For $n = 2$--$5$, let $B_n \in GL_n(\RR)$ be the matrix with $b_j^{\sigma_i}$ in each $(i, j)$-entry, where $\sigma_1, \ldots, \sigma_n$ are all the embeddings of the following $K_n$ into $\CC$ over $\QQ$, and $b_1, \ldots, b_n$ are a basis of the ring of integers of the fields $K_n$ as a $\ZZ$-module.
\begin{itemize}
\item[$\bullet$] $K_2 = \QQ(\zeta_5 + \zeta_5^{-1})$, $\zeta_5 = e^{2 \pi \sqrt{-1} / 5}$,

\item[$\bullet$] $K_3 = \QQ(\zeta_7 + \zeta_7^{-1})$, $\zeta_7 = e^{2 \pi \sqrt{-1} / 7}$,

\item[$\bullet$] $K_4 = \QQ(\sqrt{7 + 2 \sqrt{5}})$,

\item[$\bullet$] $K_5 = \QQ(\zeta_{11} + \zeta_{11}^{-1})$, $\zeta_{11} = e^{2 \pi \sqrt{-1} / 11}$,
\end{itemize}
For the above $B_n$, the following hold:
\begin{eqnarray*}
\Delta_2^\prime = \Delta^\prime(L(B_2)) = \frac{\pi}{ 2 \sqrt{5} } \approx 0.702, & &
\Delta_3^\prime = \Delta^\prime(L(B_3)) = \frac{ \sqrt{3} \pi }{ 14 } \approx 0.389, \\ 
\Delta_4^\prime \ge \Delta^\prime(L(B_4)) = \frac{ \pi^2 }{ 10\sqrt{29} } \approx 0.183, & & 
\Delta_5^\prime \ge \Delta^\prime(L(B_5)) = \frac{ 5 \sqrt{5} \pi^2 }{ 12 \cdot 11^{2} } \approx 0.076. 
\end{eqnarray*}

(For comparison, the packing density of the densest lattice packings are
$\pi/2\sqrt{3} \approx 0.907$ for $n = 2$, 
$\pi/3\sqrt{2} \approx 0.740$ \cite{Gauss1840} for $n = 3$, 
$\pi^2/16 \approx 0.617$ \cite{Korkine1873} for $n = 4$, and 
$\pi^2/15\sqrt{2} \approx 0.465$ \cite{Korkine1873} for $n = 5$.)


\item\label{item: packing densities} 
For any distinct $\varphi_1, \varphi_2 \in \RR \setminus \QQ$,
let $L_{\varphi_1, \varphi_2}$, $L_{\varphi_1, \varphi_2, \epsilon}$ be the lattices $L(B_1)$, $L(B_{\epsilon})$ generated by the column vectors of the following matrices:
$$
B_{1} =
\begin{pmatrix} 1 & -\varphi_1 \\ 1 & -\varphi_2 \end{pmatrix}, \quad
B_{\epsilon} =
\begin{pmatrix} \epsilon^{-1/2} & 0 \\ 0 & \epsilon^{1/2} \end{pmatrix}
\begin{pmatrix} 1 & -\varphi_1 \\ 1 & -\varphi_2 \end{pmatrix}.
$$
Then,
\begin{eqnarray}
\displaystyle\liminf_{ \epsilon \rightarrow +0 } \Delta \left( L_{\varphi_1, \varphi_2, \epsilon} \right) &=& \frac{\pi}{2 {\mathcal L}(\varphi_1)}, \label{eq: 1st eq of thm} \\
\Delta^\prime( L_{\varphi_1, \varphi_2} ) 
= \displaystyle\inf_{ \epsilon } \Delta \left( L_{\varphi_1, \varphi_2, \epsilon} \right)
&=& \displaystyle\frac{\pi}{2 {\mathcal M}(f)},
\label{eq: 2nd eq of thm}
\end{eqnarray}
where ${\mathcal L}(\varphi_1)$ is the Lagrange number of $\varphi_1$, 
and ${\mathcal M}(f) := \sqrt{d(f)} / m(f)$ is the element of the Markoff spectrum that corresponds to the quadratic form $f(x, y) := (x - \varphi_1 y)(x - \varphi_2 y)$.

\end{enumerate}

\end{theorem}

\begin{remark}
Eq.(\ref{eq: 1st eq of thm}) does not hold for any rational $\varphi_1$, because ${\mathcal L}(\varphi_1) = 0$ if $\varphi_1 \in \QQ$ (Corollary~1.2, \cite{Aigner2013}).
\end{remark}
\begin{remark}
When $B$ is as in (1) \IE the rows are conjugate of each other, 
the product $\abs{ \prod_{i=1}^n (b_{i1} x_1 + \cdots + b_{in} x_n) }$ is called a \textit{norm form}.
It is conjectured that for $n \ge 3$, $\lambda(B_n) > 0$ implies that $B_n$ is a norm form \cite{Cusick87}.
\end{remark}

\begin{remark}
The upper bounds $\lambda_4 \le 3/(20 \sqrt{5})$ \cite{Zilinskas41} and $\lambda_5 \le 1/57.02$ \cite{Godwin50} are also known.
Therefore, $\Delta_4^\prime \le 3 \pi^2 /(40 \sqrt{5}) \approx 0.331$ and $\Delta_5^\prime \le 5 \sqrt{5} \pi^2 /(12 \cdot 57.02) \approx 0.161$.
\end{remark}

\begin{proof}
\begin{enumerate}[(1)]
\item 
This part is an immediate consequence of Lemma~\ref{lem: lemma 1} and the known results: $\lambda_2 = 1/\sqrt{5}$ \cite{Korkine1873}, \cite{Markoff1879}, $\lambda_3 = 1/7$ \cite{Davenport38}, $\lambda_4 \ge 1/5 \sqrt{29}$ \cite{Mayer29} 
and $\lambda_5 \ge 11^{-2}$ \cite{Hunter57}.
The lower bounds for $\lambda_4, \lambda_5$ are from the smallest discriminants among all the totally real quartic and quintic fields.

\item 
As a result of Lemma~\ref{lem: lemma 1}, Eq.(\ref{eq: 2nd eq of thm}) is obtained as follows: 
$$
\Delta^\prime( L_{\varphi_1, \varphi_2 } )
= \frac{\pi \lambda(B_1) }{ 2 \abs{ \det B_1 } }
= \frac{\pi \displaystyle\inf_{0 \ne (x, y) \in \ZZ^n} f(x, y) }{2 \abs{ \varphi_1 - \varphi_2 } }
= \frac{ \pi m(f) }{ 2 \sqrt{d(f)} } = \frac{\pi }{ 2 {\mathcal M}(f) }. 
$$

Eq.(\ref{eq: 1st eq of thm}) is proved as follows; 
$\varphi_1 > 0 > \varphi_2$ may be assumed,
by replacing $B_\epsilon$ with 
$D B_\epsilon g$ for some for some diagonal $D \in GL_2(\RR)$ and $g \in GL_2(\ZZ)$.
This replaces each $\varphi_i$ by its linear fractional transformation. Hence, it does not change their Lagrange numbers.

From Proposition~\ref{prop: Selling reduced basis},  
the shortest vector of $L_{\varphi_1, \varphi_2, \epsilon}$ is equal to $p_N^{(+)} {\mathbf b}_1 - q_N^{(+)} {\mathbf b}_2$ for sufficiently small $\epsilon > 0$,
where $p_N^{(+)}/q_N^{(+)}$ is the $N$'th convergent of $\varphi_1$. 
Furthermore, the following is proved as in Lemma~\ref{lem: lemma 1}.
\begin{eqnarray*}
\displaystyle\liminf_{ \epsilon \rightarrow +0 } \min L_{\varphi_1, \varphi_2, \epsilon}
&=& 
\liminf_{\epsilon \rightarrow +0} \inf_{N > 0}
\{ \epsilon^{-1} (p_N^{(+)} - q_N^{(+)} \varphi_1)^2 + \epsilon (p_N^{(+)} - q_N^{(+)} \varphi_2)^2 \} \\
&=& 2 \liminf_{N \rightarrow \infty} \abs{ (p_N^{(+)} - q_N^{(+)} \varphi_1) (p_N^{(+)} - q_N^{(+)} \varphi_2) } \\
&=& 2 \abs{ \varphi_1 - \varphi_2 } \liminf_{N \rightarrow \infty} \abs{ q_N^{(+)} (p_N^{(+)} - q_N^{(+)} \varphi_1) } \\
&=& \frac{ 2 \abs{ \varphi_1 - \varphi_2 } }{ \limsup_{N \rightarrow \infty} \ell_N(\varphi_1) }, \quad (\CF \text{proof of Proposition1.22 of \cite{Aigner2013}}) \\
\ell_n(\varphi_1) &:=& [a_{n+1}, a_{n+2}, \ldots ] + [0, a_{n}, a_{n-1}, \ldots, a_1 ].
\end{eqnarray*}

Thus, Eq.(\ref{eq: 1st eq of thm}) is proved as follows:
$$
\displaystyle\liminf_{ \epsilon \rightarrow +0 } \Delta \left( L_{\varphi_1, \varphi_2, \epsilon} \right) =
\frac{\pi }{2 \limsup_{N \rightarrow \infty} \ell_N(\varphi_1) } = \frac{\pi }{ 2 {\mathcal L}(\varphi) }. 
$$
\qed
\end{enumerate}
\end{proof}

As an immediate consequence of Theorem~\ref{thm: theorem 1}, 
it is possible to improve 
the packing around the origin of the Vogel spiral by using the optimal lattice basis $B_2$ in Theorem~\ref{thm: theorem 1} 
(Figure~\ref{fig: Vogel spiral 2}).
More obvious examples will be provided in Figure~\ref{fig: circle packings obtained from different lattice bases}
of Section~\ref{subsec: Solutions by the method of separation of variables}. 
\begin{figure}[htbp]
\scalebox{0.45}{\includegraphics{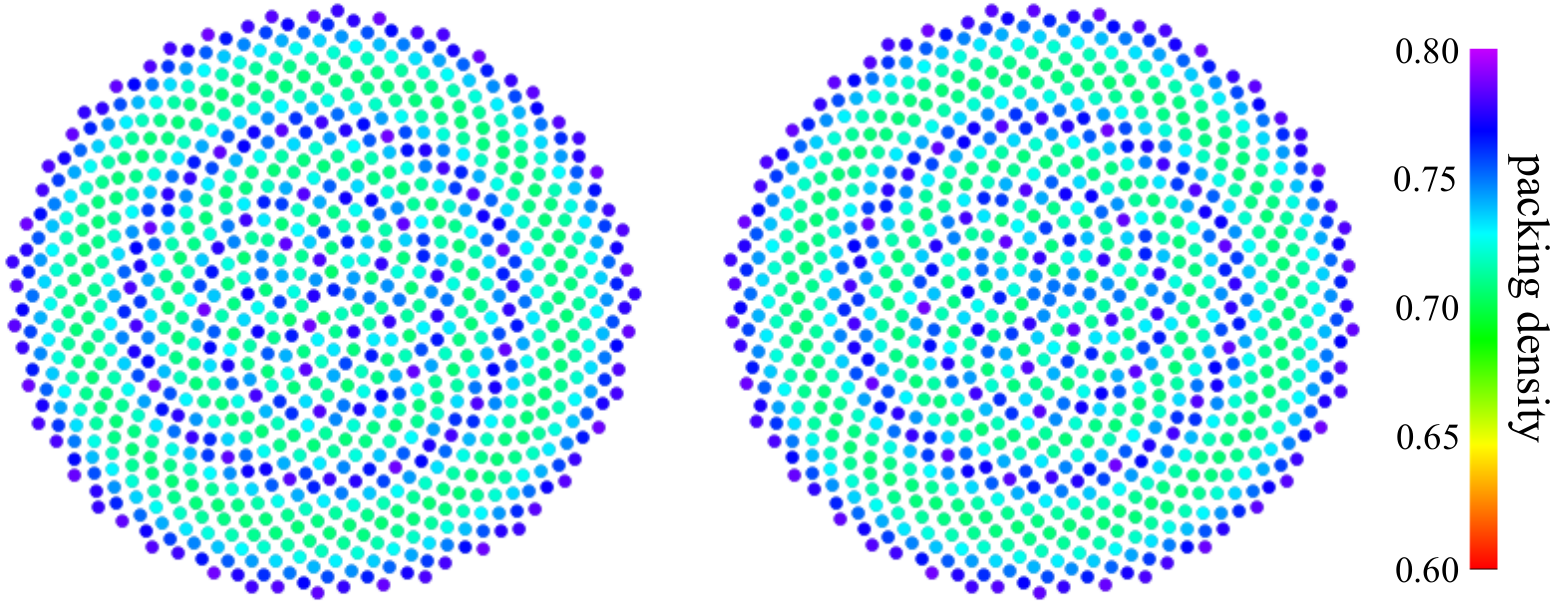}}
\caption{Original Vogel spiral (left) and a packing obtained from the lattice basis $B_2$ in Theorem~\ref{thm: theorem 1} (right).
The lattice map $f$ to use is explained in Example~\ref{exa: case of (a) and (b)}. 
}
\label{fig: Vogel spiral 2} 
\end{figure}



\begin{example}[Vogel spirals for quadratic irrationals $\varphi$ with ${\mathcal L}(\varphi) < 3$]
\label{exa: Vogel spiral for quadratic irrationals}
The $\gamma_m$ defined by Eq.(\ref{eq: Markoff quadratic irrationals}) has the Lagrange number ${\mathcal L}(\gamma_m) = \sqrt{9 - 4/m^2}$. 
Hence, the basis matrix 
${\mathcal B}_m = \begin{pmatrix} 1 & - \gamma_m \\ 1 & - \overline{\gamma}_m \end{pmatrix}$
fulfills the following for any Markoff number $m$.
$$
\Delta^\prime(L({\mathcal B}_m)) = \frac{ \pi }{ 2 {\mathcal L}(\gamma_m) } = \frac{ \pi }{ 2 \sqrt{9 - 4/m^2} } > \frac{ \pi }{ 6 } \approx 0.5236. 
$$ 

The quadratic irrationals with the smallest Lagrange numbers are
$\gamma_1 = (1 + \sqrt{5})/2 = [\overline{1}]$, 
$\gamma_2 = 1 + \sqrt{2} = [\overline{2}]$, 
$\gamma_5 = (9 + \sqrt{221})/10 = [\overline{2, 2, 1, 1}]$.
For each $m = 1, 2, 5$, the packing density of $L(D {\mathcal B}_m)$ is not less than the following value, 
regardless of the diagonal $D \in GL_2(\RR)$:
\begin{eqnarray*}
\Delta^\prime(L({\mathcal B}_1)) &=& \frac{ \pi }{ 2 \sqrt{5} } \approx 0.7025, \\
\Delta^\prime(L({\mathcal B}_2)) &=& \frac{ \pi }{ 4 \sqrt{2} } \approx 0.5554, \\
\Delta^\prime(L({\mathcal B}_5)) &=& \frac{ 5 \pi }{ 2 \sqrt{221} } \approx 0.5283. 
\end{eqnarray*}

\end{example}

%
%
\section{Lattice maps for packing of the Euclidean spaces}
\label{sec: Application to aperiodic packings}

In this section, it is assumed that $n = N$, and
$f({\mathbf x}) = (f_1({\mathbf x}), \ldots, f_n({\mathbf x})) \in C^3({\mathfrak D}, \RR^n)$ defined on open subset ${\mathfrak D} \subset \RR^n$, satisfies the properties ($\star$), ($\star\star$):
\begin{itemize}
\item[($\star$)] For any ${\mathbf x} \in {\mathfrak D}$, there are an orthogonal matrix $U({\mathbf x})$ of degree $n$
and a diagonal matrix $\Phi({\mathbf x})$ with the diagonal entries $\phi_1({\mathbf x}), \ldots, \phi_n({\mathbf x}) \ne 0$ that satisfy the following:
\begin{eqnarray*}
D f({\mathbf x}) :=
\begin{pmatrix}
\partial f_1 / \partial x_1 & 
\cdots & 
\partial f_1 / \partial x_n \\ 
\vdots & 
& 
\vdots \\ 
\partial f_N / \partial x_1 & 
\cdots & 
\partial f_N / \partial x_n \\ 
\end{pmatrix}
= 
U({\mathbf x}) \Phi({\mathbf x}).
\end{eqnarray*}

\item[($\star\star$)] $\displaystyle\prod_{i=1}^n \phi_i({\mathbf x}) = c$ for some constant $c \ne 0$.

\end{itemize}

After deriving all the PDEs for $\Phi({\mathbf x})$ and $U({\mathbf x})$ in Proposition~\ref{prop: PDE system},
a family of the PDE solutions are provided by using solutions of inviscid Burgers equation (Section~\ref{subsec: PDEs and solutions by the method of characteristics}) and separation of variables (Section~\ref{subsec: Solutions by the method of separation of variables}).

\subsection{System of PDEs and solutions provided by inviscid Burgers equation}
\label{subsec: PDEs and solutions by the method of characteristics}

First we assume only ($\star$), and 
the constraint ($\star\star$) will not be used until Example~\ref{exa: case of n=2}.
Since $f({\mathbf x}) \in C^3({\mathfrak D}, \RR^n)$ and the Jacobian matrix $D f({\mathbf x})$ fulfills $\tr{(D f)} D f = \Phi({\mathbf x})^2$,
$\phi_1({\mathbf x}), \ldots, \phi_{n}({\mathbf x})$ are functions of class $C^2$.
Let $A_k({\mathbf x})$ be the matrix defined by $A_k({\mathbf x}) := U({\mathbf x})^{-1} (U({\mathbf x}))_{x_k}$.
From $(\tr{U} U)_{x_k} = O$ and $(U_{x_k})_{x_j} = (U_{x_j})_{x_k}$, the following are obtained:
\begin{itemize}
\item[(a)] $\tr{A}_k + A_k = O$.
\item[(b)] 
$
A_j A_k - A_k A_j = (A_j)_{x_k} - (A_k)_{x_j}.
$
\end{itemize}

In particular, Eqs.(\ref{eq: 1st eq}) and (\ref{eq: 2nd eq}), which are derived only from $(\star)$ and $n = N$,
are the PDEs called Lam{\' e} equations \cite{Zakharov98}.
\begin{proposition}\label{prop: PDE system}
Let ${\mathfrak D} \subset \RR^n$ be a simply-connected open subset. 
The $\phi_1({\mathbf x}), \ldots, \phi_{n}({\mathbf x})$ that fulfill ($\star$)
for some $f \in C^3({\mathfrak D}, \RR^n)$,
are provided by solving the following PDEs:
\begin{eqnarray}
\phi_i^{-1} (\phi_i)_{x_k} (\phi_j)_{x_i}
+ \phi_k^{-1} (\phi_k)_{x_i} (\phi_j)_{x_k}
= (\phi_j)_{x_i x_k} & & (1 \le i, j, k \le n: \text{distinct}), \quad \quad \label{eq: 1st eq} \\
\left( \phi_k^{-1} (\phi_j)_{x_k} \right)_{x_k} + \left( \phi_j^{-1} (\phi_k)_{x_j} \right)_{x_j}
=
- \displaystyle\sum_{i \ne j, k}
\phi_i^{-2} (\phi_k)_{x_i} (\phi_j)_{x_i} & & (1 \le j, k \le n: \text{distinct}). \quad \quad \label{eq: 2nd eq}
\end{eqnarray}

For fixed $\Phi({\mathbf x})$, the matrix $A_k = \left( a_{ij}^{(k)} \right)$ ($k = 1, \ldots, n$) that 
satisfy the above (a) and (b) are provided by:
\begin{eqnarray}\label{eq: definition of A_k}
A_k = \tr{ {\mathbf e}_k } {\mathbf c}_k - \tr{ {\mathbf c}_k } {\mathbf e}_k,
\end{eqnarray}
where ${\mathbf c}_k := \left( \phi_1^{-1} (\phi_k)_{x_1}, \ldots, \phi_n^{-1} (\phi_k)_{x_n} \right)$
and ${\mathbf e}_k$ is the unit vector with $1$ in the $k$'th entry.
$U({\mathbf x}) \in O(n)$ is obtained by solving $U_{x_k} = U A_k$ ($k = 1, \ldots, n$).

\end{proposition} 

\begin{proof} If such an $f$ exists, from $((f_i)_{x_j})_{x_k} = ((f_i)_{x_k})_{x_j}$, 
$\phi_1({\mathbf x}), \ldots, \phi_{n}({\mathbf x})$ and $U({\mathbf x})= \begin{pmatrix} {\mathbf u}_1({\mathbf x}) & \cdots & {\mathbf u}_n({\mathbf x}) \end{pmatrix}$ satisfy:
\begin{eqnarray*}
(\phi_j)_{x_k} {\mathbf u}_j
+ \phi_j \sum_{i \ne j} a_{ij}^{(k)} {\mathbf u}_i
&=& 
(\phi_k)_{x_j} {\mathbf u}_k
+ \phi_k \sum_{i \ne k} a_{ik}^{(j)} {\mathbf u}_i \quad (1 \le j, k \le n).
\end{eqnarray*}

Since ${\mathbf u}_1, \ldots, {\mathbf u}_n$ are linearly independent over $\RR$,
the following are obtained:
\begin{eqnarray*}
\phi_j a_{ij}^{(k)} = \phi_k a_{ik}^{(j)} & & (1 \le i, j, k \le n: \text{distinct}), \\
a_{kj}^{(k)} = \phi_j^{-1} (\phi_k)_{x_j} & & (1 \le j, k \le n: \text{distinct}).
\end{eqnarray*}

It is concluded that $a_{ij}^{(k)} = 0$ for any distinct $1 \le i, j, k \le n$,
owing to $a_{ij}^{(k)} = -a_{ji}^{(k)}$ and
$$
a_{ij}^{(k)}
= \frac{\phi_k}{\phi_j} a_{ik}^{(j)}
= - \frac{\phi_k}{\phi_j} a_{ki}^{(j)}
= - \frac{\phi_k}{\phi_i} a_{kj}^{(i)}
= \frac{\phi_k}{\phi_i} a_{jk}^{(i)}
= a_{ji}^{(k)}.
$$

Thus, $A_k = \tr{ {\mathbf e}_k } {\mathbf c}_k - \tr{ {\mathbf c}_k } {\mathbf e}_k$ is obtained.
Because $A_k$ also fulfills the above (b), 
\begin{eqnarray}
& & \hspace{-10mm}
(\tr{ {\mathbf e}_j } {\mathbf c}_j - \tr{ {\mathbf c}_j } {\mathbf e}_j)
(\tr{ {\mathbf e}_k } {\mathbf c}_k - \tr{ {\mathbf c}_k } {\mathbf e}_k)
-
(\tr{ {\mathbf e}_k } {\mathbf c}_k - \tr{ {\mathbf c}_k } {\mathbf e}_k)
(\tr{ {\mathbf e}_j } {\mathbf c}_j - \tr{ {\mathbf c}_j } {\mathbf e}_j) \nonumber \\
&=& 
\phi_k^{-1} (\phi_j)_{x_k} (\tr{ {\mathbf e}_j } {\mathbf c}_k - \tr{ {\mathbf c}_k } {\mathbf e}_j)
- ({\mathbf c}_j \cdot {\mathbf c}_k) ( \tr{ {\mathbf e}_j } {\mathbf e}_k - \tr{ {\mathbf e}_k } {\mathbf e}_j )
- \phi_j^{-1} (\phi_k)_{x_j} (\tr{ {\mathbf e}_k} {\mathbf c}_j - \tr{ {\mathbf c}_j } {\mathbf e}_k) \nonumber \\
&=& 
(\tr{ {\mathbf e}_j } {\mathbf c}_j - \tr{ {\mathbf c}_j } {\mathbf e}_j)_{x_k}
- (\tr{ {\mathbf e}_k } {\mathbf c}_k - \tr{ {\mathbf c}_k } {\mathbf e}_k)_{x_j}. \label{eq: result of (b)}
\end{eqnarray}

By comparing the $(j, i)$- and $(j, k)$-components of Eq.(\ref{eq: result of (b)}), we can obtain:
\begin{eqnarray*}
(\phi_i \phi_k)^{-1} (\phi_j)_{x_k} (\phi_k)_{x_i} 
=
\left( \phi_i^{-1} (\phi_j)_{x_i} \right)_{x_k} & & (1 \le i, j, k \le n: \text{distinct}), \\
\phi_j^{-2} (\phi_j)_{x_j} (\phi_k)_{x_j} 
+ \phi_k^{-2} (\phi_j)_{x_k} (\phi_k)_{x_k} 
- {\mathbf c}_j \cdot {\mathbf c}_k & & \\
= \left( \phi_k^{-1} (\phi_j)_{x_k} \right)_{x_k}
+ \left( \phi_j^{-1} (\phi_k)_{x_j} \right)_{x_j} & & (1 \le j, k \le n: \text{distinct}).
\end{eqnarray*}
Each equation leads to Eqs.(\ref{eq: 1st eq}), (\ref{eq: 2nd eq}), respectively. 
\qed
\end{proof} 

For any constants $t_1, \ldots, t_n \ne 0$, $f(x_1, \ldots, x_n)$ fulfills ($\star$) if and only if $f(t_1 x_1, \ldots, t_n x_n)$ does.
Accordingly, $\phi_j(x_1, \ldots, x_n)$ ($j = 1, \ldots, n$) fulfill Eqs.(\ref{eq: 1st eq}) and (\ref{eq: 2nd eq})
if and only if 
$
t_j \phi_j(t_1 x_1, \ldots, t_n x_n) \quad (j = 1, \ldots, n)
$ do.
Thus, the self-similar solutions of the PDEs are provided by those that fulfill $\phi_j(x_1, \ldots, x_n) = t_j \phi_j(t_1 x_1, \ldots, t_n x_n)$ for any $0 \ne t_1, \ldots, t_n \in \RR$,
which leads to the self-similar solution $\phi_j = c_j/x_j$ ($c_j \in \RR$).
In this case, ($\star\star$) does not hold clearly. 
The map $f$ is given by
$$
f(x_1, \ldots, x_n) = U_0 \begin{pmatrix} c_1 \log x_1 \\ \vdots \\ c_n \log x_n \end{pmatrix} + {\mathbf v}_0 \quad 
(U_0 \in O(n), {\mathbf v}_0 \in \RR^N).
$$

Next, suppose that $\phi_1({\mathbf x}) = \cdots = \phi_n({\mathbf x})$ \IE $f$ is a conformal map. 
In this case, the condition ($\star\star$) only provides trivial lattice packings,
because $\phi_1 = \cdots = \phi_n$ and $\phi_1 \cdots \phi_n = c$ imply
that $\Phi({\mathbf x})$ and $U({\mathbf x})$ are constant.


\begin{example}[Case of conformal mapping without the condition ($\star\star$)]
\label{exa: Doyle spiral}
We can put $\phi({\mathbf x}) := \phi_1({\mathbf x}) = \cdots = \phi_n({\mathbf x})$.
For $n \ge 3$, 
any conformal maps are M{\" o}bius transformations of the following form as a result of Liouville's theorem.
$$
	f(x) = \frac {\alpha U (x-a)}{|x-a|^{\epsilon}} +b. \quad (a, b \in \RR^n, \alpha \in \RR, U \in O(n), \epsilon = 0, 2)
$$
From Proposition~\ref{prop: PDE system}, the PDE for $n = 2$ is $( \phi_2^{-1} (\phi_1)_{y} )_{y} + ( \phi_1^{-1} (\phi_2)_{x} )_{x} = 0$. Hence  $\phi({\mathbf x}) := \phi_1({\mathbf x}) = \phi_2({\mathbf x})$ fulfills 
$$
(\log \phi)_{x x} + (\log \phi)_{y y} = 0.
$$
Thus, $u(x, y) := \log \phi(x, y)$ is a harmonic function. 
If $v(x, y)$ is harmonic conjugate to $u(x, y)$ (\IE $u_{x} = v_{y}$ and $u_{y} = -v_{x}$), 
the following is obtained from Eq.(\ref{eq: definition of A_k}). 
\begin{eqnarray*}
A_1 &=& \begin{pmatrix} 0 & \phi^{-1} \phi_{y} \\ -\phi^{-1} \phi_{y} & 0 \end{pmatrix} = \begin{pmatrix} 0 & u_{y} \\ -u_{y} & 0 \end{pmatrix}, \\
A_2 &=& \begin{pmatrix} 0 & -\phi^{-1} \phi_{x} \\ \phi^{-1} \phi_{x} & 0 \end{pmatrix} = \begin{pmatrix} 0 & -u_{x} \\ u_{x} & 0 \end{pmatrix}, \\
U(x, y)
&=& U_0 \exp \left( \begin{pmatrix} 0 & -v \\ v & 0 \end{pmatrix} \right)
= U_0 \begin{pmatrix} \cos v & -\sin v \\ \sin v & \cos v \end{pmatrix}, \quad U_0 \in O(2).
\end{eqnarray*}

The Doyle spiral is included in this case, in which
each circle is tangent to six other circles, and the seven circles have the radii with the ratio $1, a, b, 1/a, 1/b, a/b, b/a$ ($a, b > 0$) as in Figure~\ref{fig: Doyle spiral 2}.
\begin{figure}[htbp]
\scalebox{0.4}{\includegraphics{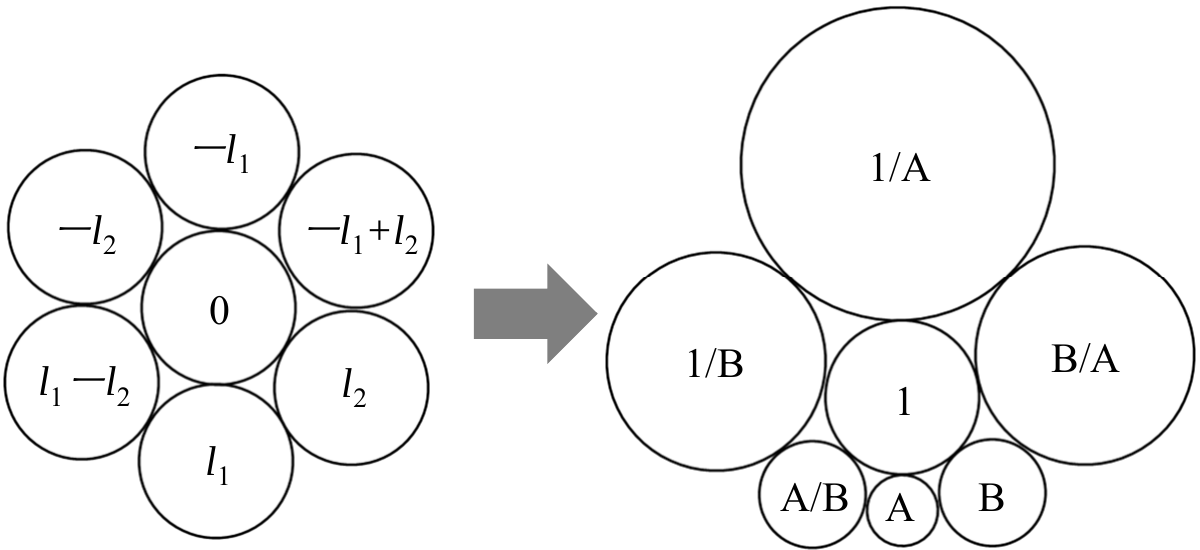}}
\caption{{\bf Left}: hexagonal lattice packing in $\RR^2$, {\bf Right}: radius ratios of adjacent circles in the Doyle spiral}
\label{fig: Doyle spiral 2} 
\end{figure}

\end{example}

The Doyle spirals are the image of a hexagonal lattice in the complex plane $\CC$ by exponential maps~\cite{Beardon94}.
In \cite{Bobenko2001}, \textit{conformally symmetric circle packings} were defined as a generalization of the Doyle spirals.
From the aspect of the golden angle method, packings with the Voronoi cells of varying volumes have been investigated in \cite{Sushida2017}, \cite{Yamagishi2017}.
In order to generate such circle packings, the determination of the radius of each circle is also necessary, which is not discussed in this article.

\begin{example}[PDE for $n = 2$]\label{exa: case of n=2} 
From $( \phi_2^{-1} (\phi_1)_{y} )_{y} + ( \phi_1^{-1} (\phi_2)_{x} )_{x} = 0$,
if ($\star\star$) is also assumed, 
$\epsilon(x, y) := c^{-1} \phi_2^2$, $\phi_1 \phi_2 = c$ 
satisfies the following equations:
\begin{eqnarray*}
\epsilon_{x x} + (\epsilon^{-1})_{y y} &=& c^{-1} ( (\phi_2^2)_{x x} + (\phi_1^2)_{y y} ) = 0, \\
A_1 = U^{-1} U_{x} &=& \begin{pmatrix} 0 & \phi_2^{-1} (\phi_1)_{y} \\ -\phi_2^{-1} (\phi_1)_{y} & 0 \end{pmatrix} = \begin{pmatrix} 0 & (\epsilon^{-1})_{y}/2 \\ -(\epsilon^{-1})_{y}/2 & 0 \end{pmatrix}, \\
A_2 = U^{-1} U_{y} &=& \begin{pmatrix} 0 & -\phi_1^{-1} (\phi_2)_{x} \\ \phi_1^{-1} (\phi_2)_{x} & 0 \end{pmatrix} = \begin{pmatrix} 0 & -\epsilon_{x}/2 \\ \epsilon_{x}/2 & 0 \end{pmatrix}.
\end{eqnarray*}

$U(x, y)$ is explicitly represented by using $\theta(x, y)$ with $\epsilon_{x} = 2\theta_{y}$ and $(\epsilon^{-1})_{y} = -2\theta_{x}$ as follows: 
$$
U(x, y)
= U_0 \exp \left( \begin{pmatrix} 0 & -\theta \\ \theta & 0 \end{pmatrix} \right)
= U_0 \begin{pmatrix} \cos \theta & -\sin \theta \\ \sin \theta & \cos \theta \end{pmatrix} \quad (U_0 \in O(2)).
$$

All the solutions of inviscid Burgers equation $\epsilon \epsilon_{x} + \epsilon_{y} = 0$
fulfill $\epsilon_{x x} + (\epsilon^{-1})_{y y} = 0$, which is also seen from the following decomposition:
\begin{eqnarray*}
\epsilon_{x x} + (\epsilon^{-1})_{y y}
=
\left( \frac{\partial}{\partial x} - \frac{\partial}{\partial y} \epsilon^{-1} \right) \left( \frac{\partial}{\partial x} + \epsilon^{-1} \frac{\partial}{\partial y} \right) \epsilon.
\end{eqnarray*}

As a result, non-trivial solutions of Eqs.(\ref{eq: 1st eq}),(\ref{eq: 2nd eq}) for dimensions $n > 2$ are obtained
by setting 
$\phi_{2k-1}^2({\mathbf x}) = c_k / \epsilon_k(x_{2k-1}, x_{2k})$, $\phi_{2k}^2({\mathbf x}) = c_k \epsilon_k(x_{2k-1}, x_{2k})$,
(and $\phi_n({\mathbf x}) = c_{(n+1)/2}$ if $n$ is odd)
for some constants $c_k$ and solutions of $\epsilon_k (\epsilon_k)_{x} + (\epsilon_k)_{y} = 0$.
\end{example}

\begin{example}[Case of inviscid Burgers equation $\epsilon_{x x} + (\epsilon^{-1})_{y y} = 0$] 
\label{exa: Case of inviscid Burgers equation}
Let $\epsilon(x, y)$ be the solution of the following inviscid Burgers equation:
\begin{eqnarray}
\begin{cases}
\epsilon \epsilon_x + \epsilon_y = 0, \\
\epsilon(t, 0) = h(t) \text{ for any } t \in I,
\end{cases}
\end{eqnarray}
where 
$h(x): I \rightarrow \RR$ is the initial condition given on an interval $I \subset \RR$.

The map $f(x, y)$ with the following Jacobian matrix, can be determined as follows:
\begin{eqnarray}
\begin{pmatrix}
(f_1)_x & (f_1)_y \\
(f_2)_x & (f_2)_y \\
\end{pmatrix}
= 
\begin{pmatrix}
\cos \theta(x, y) & -\sin \theta(x, y) \\
\sin \theta(x, y) & \cos \theta(x, y) \\
\end{pmatrix}
\begin{pmatrix}
\epsilon^{-1/2} & 0 \\
0 & \epsilon^{1/2} \\
\end{pmatrix}.
\end{eqnarray}

From $\epsilon \epsilon_x + \epsilon_y = 0$, 
\begin{eqnarray*}
2 \theta_x &=& -(\epsilon^{-1})_y = -\epsilon^{-1} \epsilon_x, \\
2 \theta_y &=& \epsilon_x = -\epsilon^{-1} \epsilon_y.
\end{eqnarray*}
Thus, $\theta = -(\log \epsilon + d)/2$ for some constant $d$.
The inviscid Burgers equation can be solved by using the characteristic equation:
\begin{eqnarray}\label{eq: characteristic equation}
\begin{cases}
q'(s) = \epsilon(q(s), s), \\
q(0) = t. 
\end{cases}
\end{eqnarray}

From Eq.(\ref{eq: characteristic equation}),
$q'(0) = \epsilon(t, 0) = h(t)$.
On the characteristic curve, $q^\prime(s) = \epsilon(q(s), s)$ is constant because 
$$
\frac{d}{ds} \epsilon(q(s), s) = q'(s) \epsilon_x(q(s), s) + \epsilon_y(q(s), s) = \epsilon (q(s), s) \epsilon_x(q(s), s) + \epsilon_y(q(s), s) = 0.
$$

Hence, $q(s) = t + h(t) s$ ($s \in \RR$) is the characteristic line,
on which $f = (f_1, f_2)$ satisfies:
\begin{eqnarray*}
\frac{d}{ds} f_1(q(s), s)
&=& q'(s) (f_1)_{x_1} + (f_1)_{x_2}
= \epsilon^{1/2} \left( \cos \frac{\log \epsilon + d}{2} + \sin \frac{\log \epsilon + d}{2} \right), \\
\frac{d}{dt} f_2(q(s), s)
&=& q'(s) (f_2)_{x_1} + (f_2)_{x_2} = \epsilon^{1/2} \left( \cos \frac{\log \epsilon + d}{2} - \sin \frac{\log \epsilon + d}{2} \right).
\end{eqnarray*}
Without loss of generality, $d = -\pi/2$ may be assumed.
In this case, 
for any $(x, y) \in \RR^2$ and $t \in I$ that satisfies $t + h(t) y = x$, $f = (f_1, f_2)$ is given by
\begin{eqnarray*}
f(x, y)
&=& 
f(t, 0)
+ y \sqrt{ 2 h(t) } 
\begin{pmatrix}
\sin \frac{\log h(t)}{2} \\
\cos \frac{\log h(t)}{2}
\end{pmatrix}, \\
f(t, 0)
&=& 
\begin{pmatrix}
\int (f_1)_{x_1} (t, 0) dt \\
\int (f_2)_{x_1}(t, 0) dt
\end{pmatrix}
=
\begin{pmatrix}
\int \frac{ 1 }{ \sqrt{ h(t)} } \cos \frac{\log h(t) - \pi/2}{2} dt \\
- \int \frac{ 1 }{ \sqrt{ h(t)} } \sin \frac{\log h(t) - \pi/2}{2} dt
\end{pmatrix}.
\end{eqnarray*}
In particular, if $f_1(t, 0) \cos \frac{\log h(t)}{2} \ne f_2(t, 0) \sin \frac{\log h(t)}{2}$,
\begin{eqnarray*}
f(x, y)
&=& (f_1(t, 0) \cos \frac{\log h(t)}{2} - f_2(t, 0) \sin \frac{\log h(t)}{2})
\begin{pmatrix}
Y & 1 \\
-1 & Y \\
\end{pmatrix}
\begin{pmatrix}
\sin \frac{\log h(t)}{2} \\
\cos \frac{\log h(t)}{2}
\end{pmatrix}, \\
Y &=& \frac{ y \sqrt{2 h(t)} + f_1(t, 0) \sin \frac{\log h(t)}{2} + f_2(t, 0) \cos \frac{\log h(t)}{2} }{ f_1(t, 0) \cos \frac{\log h(t)}{2} - f_2(t, 0) \sin \frac{\log h(t)}{2} }. 
\end{eqnarray*}
Otherwise, by using $g(t) := f_1(t, 0) / \sin \frac{\log h(t)}{2} = f_2(t, 0) / \cos \frac{\log h(t)}{2}$,
\begin{eqnarray*}
f(x, y)
&=& 
Y
\begin{pmatrix}
\sin \frac{\log h(t)}{2} \\
\cos \frac{\log h(t)}{2}
\end{pmatrix}, \quad
Y = y \sqrt{ 2 h(t) } + g(t).
\end{eqnarray*}
Both correspond to the case of Eq.(\ref{eq: f with self similarity}), 
if $f$ is regarded as the function of $(t, Y)$.
\end{example}


\subsection{A family of solutions obtained by separation of variables}
\label{subsec: Solutions by the method of separation of variables}

Another family of solutions of $\epsilon_{x x} + (\epsilon^{-1})_{y y} = 0$ can be obtained by separation of variables.
If we put $\epsilon(x, y) = F(x) / G(y)$,
$F''(x) / G(y) = - G''(y) / F(x)$ is obtained from
$\epsilon_{x x} = - (\epsilon^{-1})_{y y}$.
Hence, for some constant $\alpha$,
$$
F(x) F''(x) = \alpha, \quad G(y) G''(y) = -\alpha.
$$
If $\alpha \ne 0$, $(F'(x))^2 = 2 \alpha \log F(x) + d_1$, $(G'(y))^2 = -2 \alpha \log G(y) + d_2$.
The solutions $F(x)$ and $G(y)$ are functions represented by incomplete gamma functions in general.
Only the special case of $\alpha = 0$ is discussed below to obtain 2D and 3D analogues of the Vogel spiral.

If $\alpha = 0$, we have $F(x) = c_1 x + d_1$ and $G(y) = c_2 y + d_2$. By translating the $x$ and $y$-coordinates, 
it may be assumed that $\epsilon(x, y) = F(x)/G(y)$ is equal to either of
(a) $\epsilon = 1/\beta y$, 
(b) $\epsilon = \beta x$,
(c) $\epsilon = \beta x/y$. 
In case of (a)--(c), 
$\epsilon$ satisfies $\epsilon \epsilon_x + \epsilon_y = 0$ if and only if (c) and $\beta = 1$.

The following examples explains each case. 
The case (b) may be omitted, because it can be obtained from (a) by exchanging $x$ and $y$, and $f_1$ and $f_2$. 
The case (a) contains the Vogel spiral as a special case.
\begin{example}[(a) $\epsilon(x, y) = 1/\beta y$: packing of a disk]
\label{exa: case of (a) and (b)}
In this example, $\varphi$ is always set to $1/(1 + \gamma_1) = (3 - \sqrt{5})/2$.
Without loss of generality, $\beta > 0$ may be assumed. 

From $\epsilon(x, y) = 1/\beta y$, the map $f$ is as follows, up to similarity transformations.
$$
f(x, y) = \sqrt{y} \begin{pmatrix} \cos (\beta x/2) \\ \sin (\beta x/2) \end{pmatrix},
$$
which is injective on the following ${\mathfrak D}$:
$$
{\mathfrak D} := \left\{ (x, y) \in \RR^2: 0 \le x < 4 \pi/\beta,\ 0 < y < M \right\}.
$$

The following
are considered as the basis matrix $B$ of $L$.
$$
\text{(i) }
\begin{pmatrix}
1 & -\varphi \\ 0 & -1
\end{pmatrix}, \quad
\text{(ii) }
\begin{pmatrix}
1 & -\varphi \\
1 & -\overline{\varphi} \\
\end{pmatrix}, \quad
\text{(iii) }
\begin{pmatrix}
0 & -1 \\
1 & -\overline{\varphi} \\
\end{pmatrix},
$$
Although it was proved in Theorem~\ref{thm: theorem 1} that the lattice basis in (ii) is optimal with respect to the stability of the packing density,
the mapped lattice $L$ also needs to contain a vector very close to $\tr{ (4 \pi / \beta, 0) }$
in order to smoothly connect the spiral patterns near the half-lines $x = 0$ and $x = 4 \pi/\beta$ (see Figure~\ref{fig: circle packings obtained from different lattice bases}).

\begin{enumerate}[(i)]
\item If $s := 4 \pi / \beta$ is an integer, then $\tr{ (s, 0) }$ is precisely contained in $L(B)$ in this case. 
If $s = 1$, $f(L(B) \cap {\mathfrak D})$ is the same as the Vogel spiral. 
$$
f(L(B) \cap {\mathfrak D}) := \left\{ \sqrt{n} (\cos (2 \pi n \varphi /s), \sin (2 \pi n \varphi/s)): n \in \ZZ,\ 0 < n < M \right\}.
$$
However, as shown in (i) of Figure~\ref{fig: circle packings obtained from different lattice bases}, 
the points $f(L(B) \cap {\mathfrak D})$ are not well distributed around the origin for large $s$,
because $\epsilon = s / 4 \pi y$ is not sufficiently small (\CF (2) of Theorem~\ref{thm: theorem 1}).
The non-uniformity can be avoided by using the basis (ii) instead. 

\item In this case, $L(B)$ does not contain any vectors of the form $\tr{ (s, 0) }$.
However, the $n$'th convergent $p_n^{(-)}/q_n^{(-)}$ of $-1/\overline{\varphi}$ fulfills:
\begin{eqnarray*}
\begin{pmatrix}
1 & -\varphi \\
1& -\overline{ \varphi }
\end{pmatrix}
\begin{pmatrix}
q_n^{(-)} \\
-p_n^{(-)} 
\end{pmatrix} 
&=& 
\begin{pmatrix}
q_n^{(-)} + p_n^{(-)} \varphi \\
q_n^{(-)} + p_n^{(-)} \overline{\varphi} \\
\end{pmatrix} \\
&=& 
\begin{pmatrix}
2 q_n^{(-)} + p_n^{(-)} (\varphi + \overline{\varphi}) \\
0
\end{pmatrix}
+ q_n^{(-)} \overline{ \varphi } \left( \overline{ \varphi }^{-1} + \frac{ p_n^{(-)} }{ q_n^{(-)} } \right)
\begin{pmatrix}
-1 \\
1
\end{pmatrix}.
\end{eqnarray*}

Therefore, by setting $s$ to one of $\abs{ 2 q_n^{(-)} + p_n^{(-)} (\varphi + \overline{\varphi}) }$ ($n \ge 0$), 
it is possible to glue the spiral patterns near the boundary of ${\mathfrak D}$ apparently smoothly, as seen in (ii) of Figure~\ref{fig: circle packings obtained from different lattice bases}.
This technique of setting $s$ to a special value will be also used in the other examples.
For the given parameters, the discontinuity in spirals could be completely removed by violating the assumptions ($\star$) and ($\star\star$) very slightly.

\item As in case (ii), although $L(B)$ does not contain any vectors $\tr{ (s, 0) }$, the boundary problem can be avoided,
by putting $s = \abs{ p_n^{(-)} }$, because 
\begin{eqnarray*}
\begin{pmatrix}
0 & -1 \\
1& -\overline{ \varphi }
\end{pmatrix}
\begin{pmatrix}
q_n^{(-)} \\
-p_n^{(-)} 
\end{pmatrix} 
&=& 
\begin{pmatrix}
p_n^{(-)} \\
q_n^{(-)} + p_n^{(-)} \overline{\varphi} \\
\end{pmatrix} 
\approx
\begin{pmatrix}
p_n^{(-)} \\
0
\end{pmatrix}.
\end{eqnarray*}
The packing becomes sparse at the coordinates farther from the origin, as seen in (iii) of Figure~\ref{fig: circle packings obtained from different lattice bases}.
The number of spines is equal to the parameter $s = \abs{ p_n^{(-)} }$.
\end{enumerate}
\begin{figure}[htbp]
\scalebox{0.45}{\includegraphics{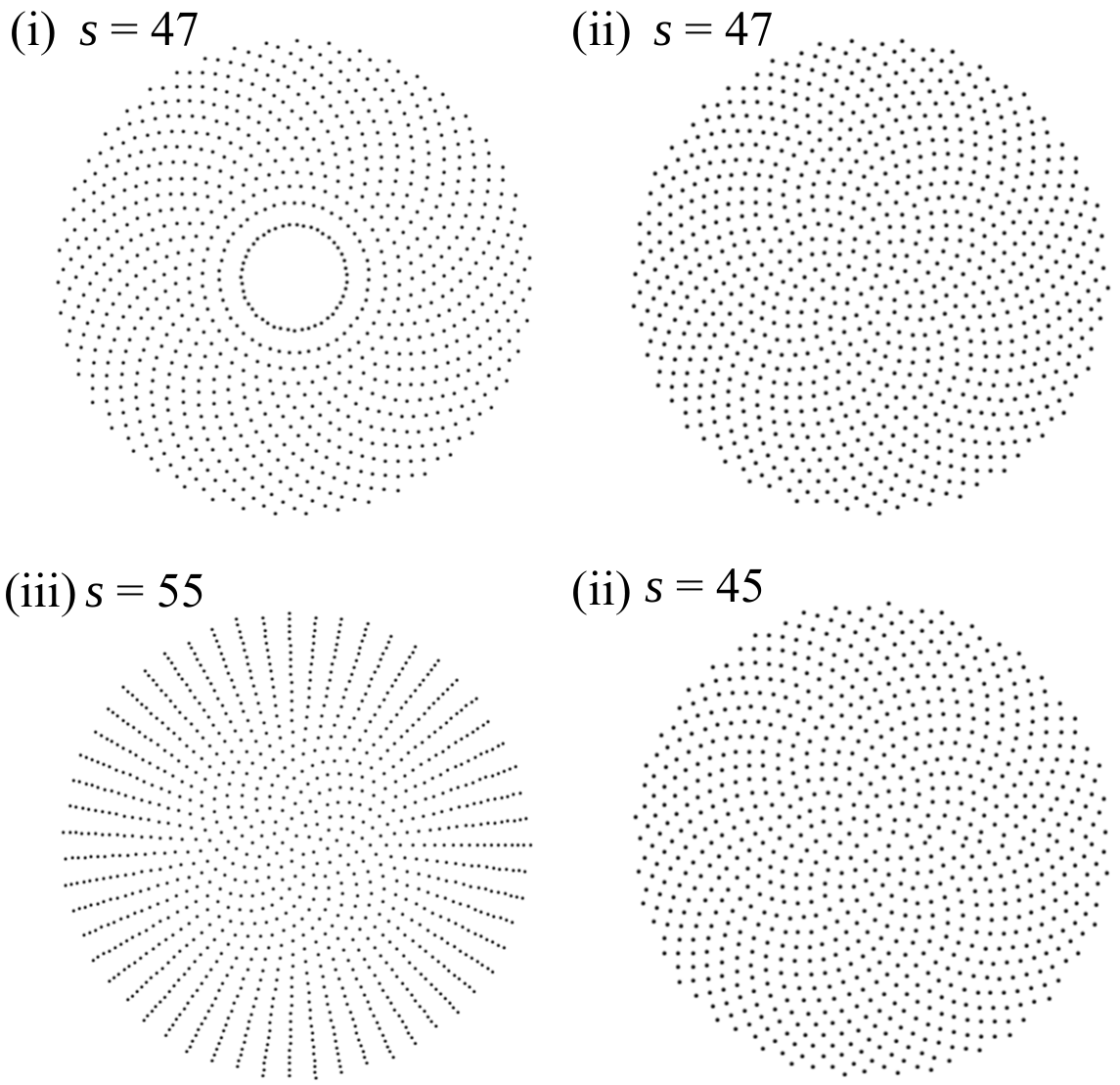}}
\caption{Packings in a disk obtained from the lattice bases (i)--(iii) of Example~\ref{exa: case of (a) and (b)}.
The parameter $s$ is set to
(i), (ii) $s = 2 q_9^{(-)} + (\varphi + \overline{\varphi}) p_9^{(-)} = 47$, where $p_9^{(-)} = -21$, $q_9^{(-)} = 55$ are the ninth convergent of $-1/\overline{\varphi} = (-3+\sqrt{5})/2$, and (iii) $s = -p_{11}^{(-)} = 55$.
As for (ii), the case of $s = 45 \ne 2 q_n^{(-)} + (\varphi + \overline{\varphi}) p_n^{(-)}$ ($n \in \ZZ_{>0}$) is also presented.
The arrow indicates that the spirals are not smoothly connected around the $x>0$ part of the $x$-axis.
} 
\label{fig: circle packings obtained from different lattice bases}
\end{figure}
\end{example}

It is known that the parastichies in the Vogel spiral are the Fermat spiral.
In fact, the image of $(x, y)$ by the following $f$ has the polar coordinate $(r, \theta) = (\sqrt{y}, \beta x/2)$.
$$
f(x, y) = \sqrt{y} \begin{pmatrix} \cos (\beta x/2) \\ \sin (\beta x/2) \end{pmatrix}.
$$

If $x, y$ is colinear, $(r, \theta) = (\sqrt{y}, \beta x/2)$ satisfies a linear equation $r^2 = a + b \theta$ for some constants $a$ and $b$,
which is an equation of the Fermat spiral.

In the following case of (c), the image of any line $y = a x$ passing through the origin by the map $f$ in Eq.(\ref{eq: f in case of (c)}), is a logarithmic spiral, because $\log r$ and $\theta$ fulfill a linear equation.
\begin{example}[(c) $\epsilon(x, y) = \beta x /y$, case of logarithmic spirals]
\label{exa: i-c solutions} 
The map $f$ is as follows, up to similarity transformations:
\begin{eqnarray}\label{eq: f in case of (c)}
f(x, y) =
\sqrt{x y} 
\begin{pmatrix}
\cos \theta(x, y) \\
\sin \theta(x, y) \\
\end{pmatrix}, \quad
\theta(x, y) = - \displaystyle\frac{\beta^{-1}}{2} \log \abs{ x } + \frac{\beta}{2} \log \abs{ y }.
\end{eqnarray}

The map $f$ is injective on the ${\mathfrak D}$:
\begin{eqnarray}\label{eq: f in case of (c) 2}
{\mathfrak D} := \left\{ (x, y) \in \RR^2: 0 < \log x \le \frac{4 \pi}{\beta + \beta^{-1}},\ 0 < y \le M \right\}.
\end{eqnarray}


As in the previous example, it is necessary to set $s := \exp( 4 \pi/(\beta + \beta^{-1}) )$ to a positive integer $s = 2 q_n^{(-)} + (\varphi + \overline{\varphi}) p_n^{(-)}$.
Since $X^2 - (4 \pi / \log s) X + 1 = 0$ has a real root $\beta$,
$1 \le s \le e^{2 \pi} \approx 535.5$ is also required.
\begin{figure}[htbp]
\scalebox{0.48}{\includegraphics{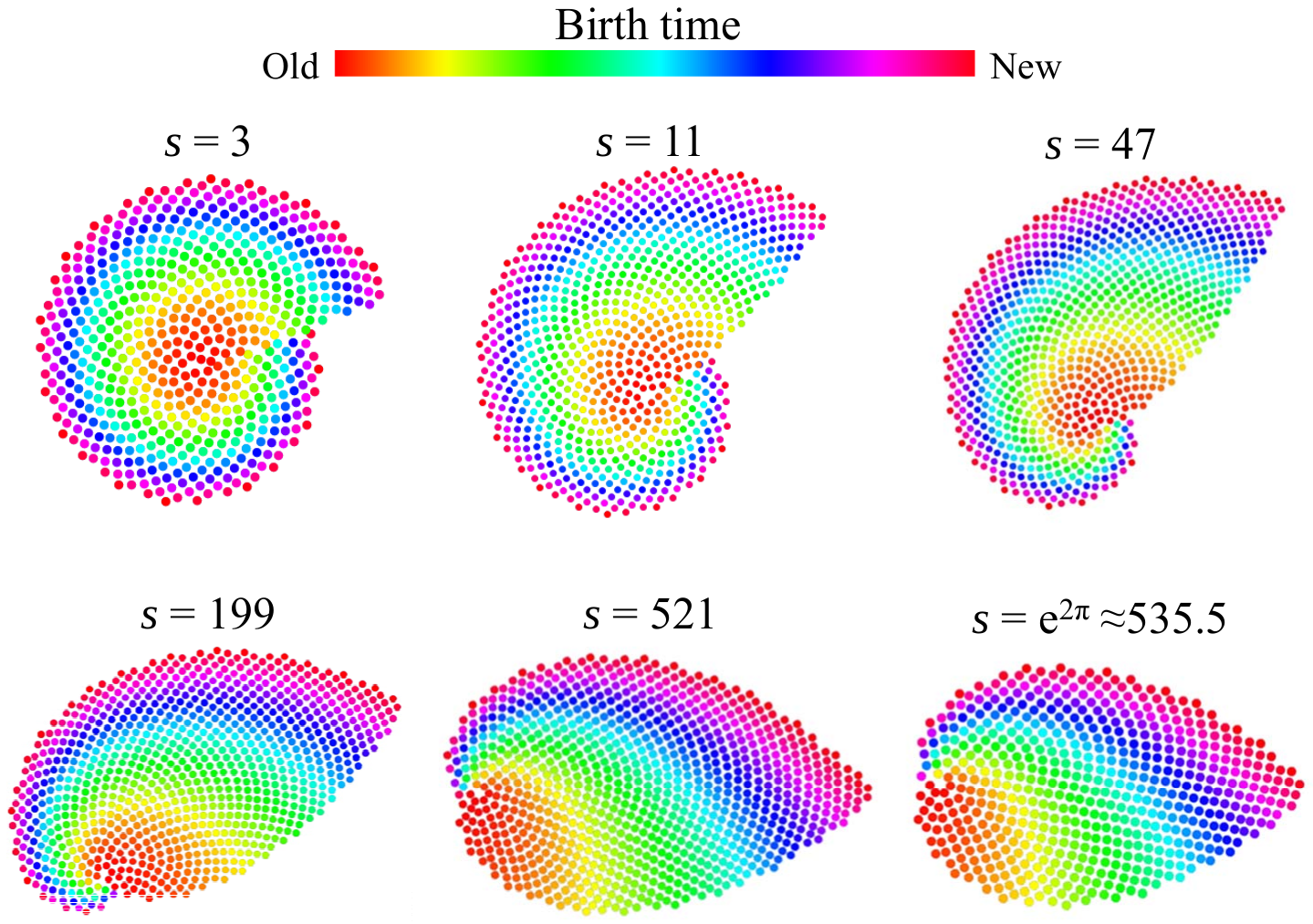}}
\caption{Packing of planes with logarithmic spirals. 
Each point is colored according to the $y$-value (birth time) of its preimage (\CF Eq.(\ref{eq: f in case of (c) 2})). The points with the same $y$-value form the identical shape, regardless of $0 < y \le M$. This self-similarity explains their biological shapes. The last $s = e^{2\pi}$ is also the case of an inviscid Burgers solution.
}
\label{fig: i-c solutions} 
\end{figure}
\end{example}

Proposition \ref{prop: solutions when A_3 = O} is a case of dimension $n=3$.
\begin{proposition}\label{prop: solutions when A_3 = O}
If $n = 3$ and $A_3 = 0$, the solutions $\phi_1({\mathbf x}), \phi_2({\mathbf x}), \phi_3({\mathbf x})$ for the system of PDEs in Eqs.(\ref{eq: 1st eq}),~(\ref{eq: 2nd eq}) and ($\star\star$) are given by 
\begin{eqnarray*}
\phi_1({\mathbf x})
&=&
c^{1/3} (d^3 + 3 d_2 x_3)^{1/3} \epsilon^{-1/2}(x_1, x_2), \\
\phi_2({\mathbf x})
&=&
c^{1/3} (d^3 + 3 d_2 x_3)^{1/3} \epsilon^{1/2}(x_1, x_2), \\
\phi_3({\mathbf x})
&=&
c^{1/3} (d^3 + 3 d_2 x_3)^{-2/3}
\end{eqnarray*}
for some constants $d, d_2$ and $\epsilon(x_1, x_2)$ with $\epsilon_{x_1 x_1} + (\epsilon^{-1})_{x_2 x_2} = - 2 d_2^2$.

\end{proposition}

\begin{proof}
From $A_3 = U^{-1} U_{x_3} = O$, $U$ is independent of $x_3$.
From (\ref{eq: definition of A_k}), $(\phi_3)_{x_1} = (\phi_3)_{x_2} = 0$ also holds.
Eq.(\ref{eq: 2nd eq}) implies that 
$\left( \phi_3^{-1} (\phi_1)_{x_3} \right)_{x_3}
= \left( \phi_3^{-1} (\phi_2)_{x_3} \right)_{x_3} = 0$.
Thus, if we put $G(x_3) := \int_0^{x_3} \phi_3(x_1, x_2, x) d x$,
the equation 
$$
\phi_i({\mathbf x}) = F_i(x_1, x_2) G(x_3) + H_i(x_1, x_2) \quad (i = 1, 2).
$$ 
holds for some $F_i(x_1, x_2)$ and $H_i(x_1, x_2)$.

From $\phi_1 \phi_2 \phi_3 = c$, $(\phi_1 \phi_2)_{x_j} = 0$ holds for $j = 1, 2$.
Thus, for some $\alpha_1, \alpha_2, \beta_1, \beta_2 \in \RR$ and $\tilde{F}(x_1, x_2)$,
\begin{eqnarray*}
	\phi_1({\mathbf x}) &=& (\alpha_1 G(x_3) + \beta_1) \tilde{F}(x_1, x_2), \\
	\phi_2({\mathbf x}) &=& (\alpha_2 G(x_3) + \beta_2)/\tilde{F}(x_1, x_2).
\end{eqnarray*}
The following equations are obtained from (\ref{eq: 1st eq}) and (\ref{eq: 2nd eq}). 
\begin{eqnarray}
(\phi_1)_{x_2 x_3}
&=&
\phi_2^{-1} (\phi_2)_{x_3} (\phi_1)_{x_2}, 
\\
(\phi_2)_{x_1 x_3}
&=&
\phi_1^{-1} (\phi_1)_{x_3} (\phi_2)_{x_1}, 
\\
\left( \phi_1^{-1} (\phi_2)_{x_1} \right)_{x_1} + \left( \phi_2^{-1} (\phi_1)_{x_2} \right)_{x_2}
&=&
- 
\phi_3^{-2} (\phi_1)_{x_3} (\phi_2)_{x_3}. \label{eq: third eq in example 2}
\end{eqnarray}

The first two equalities imply that
$\left( \log (\phi_1)_{x_2} \right)_{x_3}
=
(\log \phi_2)_{x_3}$ and
$\left( \log (\phi_2)_{x_1} \right)_{x_3}
=
(\log \phi_1)_{x_3}$.
Thus, the following functions are independent of $x_3$. 
$$
\phi_2^{-1} (\phi_1)_{x_2}
= \displaystyle\frac{1}{2}\frac{\alpha_1 G(x_3) + \beta_1}{\alpha_2 G(x_3) + \beta_2} (\tilde{F}^2)_{x_2}, \quad
\phi_1^{-1} (\phi_2)_{x_1}
= \displaystyle\frac{1}{2}\frac{\alpha_2 G(x_3) + \beta_2}{\alpha_1 G(x_3) + \beta_1} (\tilde{F}^{-2})_{x_1}.
$$

Hence, it is possible to choose 
$r_1, r_2 \ne 0$, $d, d_2 \in \RR$ and $\tilde{F}(x, y)$ so that $\phi_1 = r_1 (d + d_2 G(x_3)) \tilde{F}(x_1, x_2)$ and $\phi_2 = r_2 (d + d_2 G(x_3)) / \tilde{F}(x_1, x_2)$.
The following equations are obtained from $\phi_1 \phi_2 \phi_3 = r_1 r_2 (d + d_2 G(x_3))^2 \phi_3(x_3) = c$, $\phi_3(x_3) = G^\prime(x_3)$, $G(0) = 0$ and Eq.(\ref{eq: third eq in example 2}):
\begin{eqnarray*}
r_1 r_2 (d + d_2 G(x_3))^3 &=& r_1 r_2 d^3 + 3 c d_2 x_3, \\ 
\frac{r_2}{2 r_1} (\tilde{F}^{-2})_{x_1 x_1}
+ \frac{r_1}{2 r_2} (\tilde{F}^2)_{x_2 x_2}
&=&
- r_1 r_2 d_2^2.
\end{eqnarray*}

By putting $\epsilon(x_1, x_2) := (r_2/r_1) \tilde{F}(x_1, x_2)^{-2}$, 
$\epsilon_{x_1 x_1} + (\epsilon^{-1})_{x_2 x_2} =
- 2 (r_1 r_2) d_2^2$ is obtained in addition to the following.
\begin{eqnarray*}
\phi_1({\mathbf x}) 
&=&
(r_1 r_2)^{1/2} (d^3 + 3 (c d_2/r_1 r_2) x_3)^{1/3} \epsilon^{-1/2}(x_1, x_2), \\
\phi_2({\mathbf x}) 
&=&
(r_1 r_2)^{1/2} (d^3 + 3 (c d_2/r_1 r_2) x_3)^{1/3} \epsilon^{1/2}(x_1, x_2), \\
\phi_3({\mathbf x}) 
&=&
c (r_1 r_2)^{-1} (d + d_2 G(x_3))^{-2}
=
c (r_1 r_2)^{-1} (d^3 + 3 (c d_2/r_1 r_2) x_3)^{-2/3}.
\end{eqnarray*}

The statement is proved if $d, d_2$ are replaced by $c^{1/3} (r_1 r_2)^{-1/2} d$, $(r_1 r_2)^{-1/2} d_2$.
\qed
\end{proof}

Separation of variables can also be used to obtain
a family of solutions of $\epsilon_{x_1 x_1} + (\epsilon^{-1})_{x_2 x_2} = - 2 d_2^2$ for $d_2 \ne 0$.
A packing of a ball (\IE 3D analogue of the Vogel spiral), is obtained as a result;
if we put $\epsilon(x_1, x_2) = F(x_1) / G(x_2)$, 
$\epsilon_{x_1 x_1} + (\epsilon^{-1})_{x_2 x_2} + 2d_2^2 = 0$ implies:
$$
F(x_1) F''(x_1) + G(x_2) G''(x_2) + 2 d_2^2 F(x_1) G(x_2) = 0.
$$

Hence $F(x_1)$ or $G(x_2)$ must be a constant function, and either of 
(a) $\epsilon = d_0 + d_1 x_1 - (d_2 x_1)^2$ or (b) $\epsilon = 1/\{ d_0 + d_1 x_2 - (d_2 x_2)^2 \}$ holds.
As the part (b) can be obtained by swapping the roles of $x_1$ and $x_2$, and $f_1$ and $f_2$ in (a), we will discuss only the case (a) below.
\begin{example}[Packing of a ball, 3D Vogel spiral]
\label{exa: 3D Vogel spiral}
From $\epsilon = d_0 + d_1 x_1 - (d_2 x_1)^2$, $d_2 \ne 0$ and $A_k = \tr{ {\mathbf e}_k } {\mathbf c}_k - \tr{ {\mathbf c}_k } {\mathbf e}_k$,
$$
A_1 
=
\begin{pmatrix}
0 & 0 & d_2 \epsilon^{-1/2} \\
0 & 0 & 0 \\
-d_2 \epsilon^{-1/2} & 0 & 0 \\
\end{pmatrix}, \quad
A_2 
=
\begin{pmatrix}
0 & -\epsilon_{x_1} /2 & 0 \\
\epsilon_{x_1} /2 & 0 & d_2 \epsilon^{1/2} \\
0 & -d_2 \epsilon^{1/2} & 0 \\
\end{pmatrix}, \quad A_3 = O.
$$

The following $D_j, V_j$ ($j = 1, 2$) provide diagonalizations $A_1 = V_1 D_1 V_1^*$, $A_2 = V_2 D_2 V_2^*$ of $A_1$ and $A_2$.
\begin{eqnarray*}
D_1 
&=& d_2 \epsilon^{-1/2}
\begin{pmatrix}
i & 0 & 0 \\
0 & -i & 0 \\
0 & 0 & 0
\end{pmatrix}, \quad
V_1
= \frac{1}{ \sqrt{2} }
\begin{pmatrix}
-i & i & 0 \\
0 & 0 & \sqrt{2} \\
1 & 1 & 0 \\
\end{pmatrix}, \\
D_2 
&=& \sqrt{ d_1^2/4 + d_0 d_2^2 }
\begin{pmatrix}
i & 0 & 0 \\
0 & -i & 0 \\
0 & 0 & 0
\end{pmatrix}, \\
V_2
&=& \frac{1}{ \sqrt{2 (d_1^2/4 + d_0 d_2^2) } }
\begin{pmatrix}
\epsilon_{x_1}/2 & \epsilon_{x_1}/2 & -\sqrt{2} d_2 \epsilon^{1/2} \\
-i \sqrt{ d_1^2/4 + d_0 d_2^2 } & i \sqrt{ d_1^2/4 + d_0 d_2^2 } & 0 \\
d_2 \epsilon^{1/2} & d_2 \epsilon^{1/2} & \epsilon_{x_1}/\sqrt{2} \\
\end{pmatrix}.
\end{eqnarray*}

From $(V_j)_{x_j} = O$ 
and $U_{x_j} = U A_j$ ($j = 1, 2$), $U({\mathbf x})$ satisfies
$(U V_j)_{x_j} = U V_j D_j$. Hence, for some $U_0 \in O(3)$,
\begin{small}
\begin{eqnarray*}
U({\mathbf x})
= \frac{1}{ \sqrt{d_1^2/4 + d_0 d_2^2} } 
U_0
\begin{pmatrix}
1 & 0 & 0 \\
0 & \cos \sqrt{d_1^2/4 + d_0 d_2^2} x_2 & \sin \sqrt{d_1^2/4 + d_0 d_2^2} x_2 \\
0 & -\sin \sqrt{d_1^2/4 + d_0 d_2^2} x_2 & \cos \sqrt{d_1^2/4 + d_0 d_2^2} x_2 \\
\end{pmatrix}
\begin{pmatrix}
d_2 \epsilon^{1/2} & 0 & -\epsilon_{x_1} /2 \\
0 & \sqrt{d_1^2/4 + d_0 d_2^2} & 0 \\
\epsilon_{x_1} /2 & 0 & d_2 \epsilon^{1/2} \\
\end{pmatrix}.
\end{eqnarray*}
\end{small}

The Jacobian matrix of the map $f$ is provided as $U({\mathbf x}) \Phi({\mathbf x})$.
Hence, 
for some ${\mathbf v}_0 \in \RR^3$,
\begin{eqnarray*}
f({\mathbf x})
&=& 
\frac{ c^{1/3} \sqrt{d_1^2/4 + d_0 d_2^2} }{d_2} \left((d^3 + 3 d_2 x_3)^{1/3} U_0
\begin{pmatrix}
-\epsilon_{x_1} /2 \\
d_2 \epsilon^{1/2} \sin \sqrt{d_1^2/4 + d_0 d_2^2} x_2 \\
d_2 \epsilon^{1/2} \cos \sqrt{d_1^2/4 + d_0 d_2^2} x_2 \\
\end{pmatrix} + {\mathbf v}_0 \right) \\
&\propto& (d^3 + 3 d_2 x_3)^{1/3} U_0
\begin{pmatrix}
- d_1/2 + d_2^2 x_1 \\
\sqrt{d_1^2/4 + d_0 d_2^2 - (d_2^2 x_1 - d_1/2)^2 } \sin \sqrt{d_1^2/4 + d_0 d_2^2} x_2 \\
\sqrt{d_1^2/4 + d_0 d_2^2 - (d_2^2 x_1 - d_1/2)^2 } \cos \sqrt{d_1^2/4 + d_0 d_2^2} x_2 \\
\end{pmatrix} + {\mathbf v}_0.
\end{eqnarray*}

For constant $r > 0$, 
by putting $s := \frac{1}{\sqrt{d_1^2/4 + d_0 d_2^2}}$, $U_0 = I$ and ${\mathbf v}_0 = 0$, and replacing $x_1$, $x_2$, $x_3$ by $(x_2/r s + d_1/2)/d_2^2$, $2 \pi x_1$, $(x_3 - d^3)/3 d_2$ respectively,
\begin{eqnarray*}
f({\mathbf x})
&\propto& x_3^{1/3}
\begin{pmatrix}
x_2 \\
\sqrt{ r^2 - x_2^2 } \sin (2 \pi x_1/s) \\
\sqrt{ r^2 - x_2^2 } \cos (2 \pi x_1/s) \\
\end{pmatrix}.
\end{eqnarray*}

The above $f$ is injective on the ${\mathfrak D}$, and maps ${\mathfrak D}$ onto a ball of radius $r R$: 
$$
{\mathfrak D} := \{ (x_1, x_2, x_3) \in \RR^3: 0 \le x_1 < s,\ -r < x_2 < r,\ 0 < x_3 < R \}.
$$

$B_3$ of Theorem~\ref{thm: theorem 1} can be used as the lattice basis. 
Figures~\ref{fig: 3D Vogel spiral surface}~and~\ref{fig: 3D Vogel spiral} present the surface pattern and the cross-sections of the 3D Vogel spiral for the parameters $s = 1$ and $r = R = 1000$.
\begin{figure}[htbp]
\scalebox{0.5}{\includegraphics{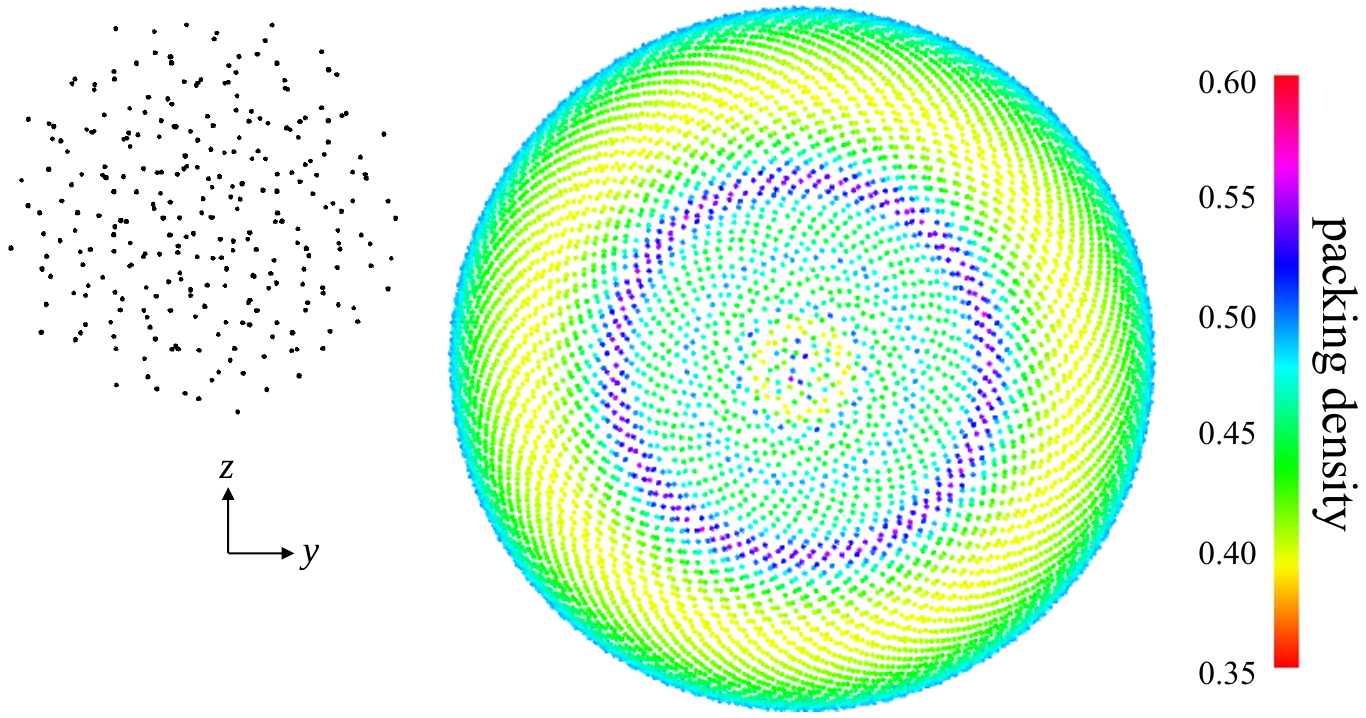}}
\caption{Point distribution around the origin (left) and the pattern on the semisphere (right) of the 3D Vogel spiral for the parameters $s = 1$, $r = R = 1000$. }
\label{fig: 3D Vogel spiral surface} 
\end{figure}
\begin{figure}[htbp]
\scalebox{0.55}{\includegraphics{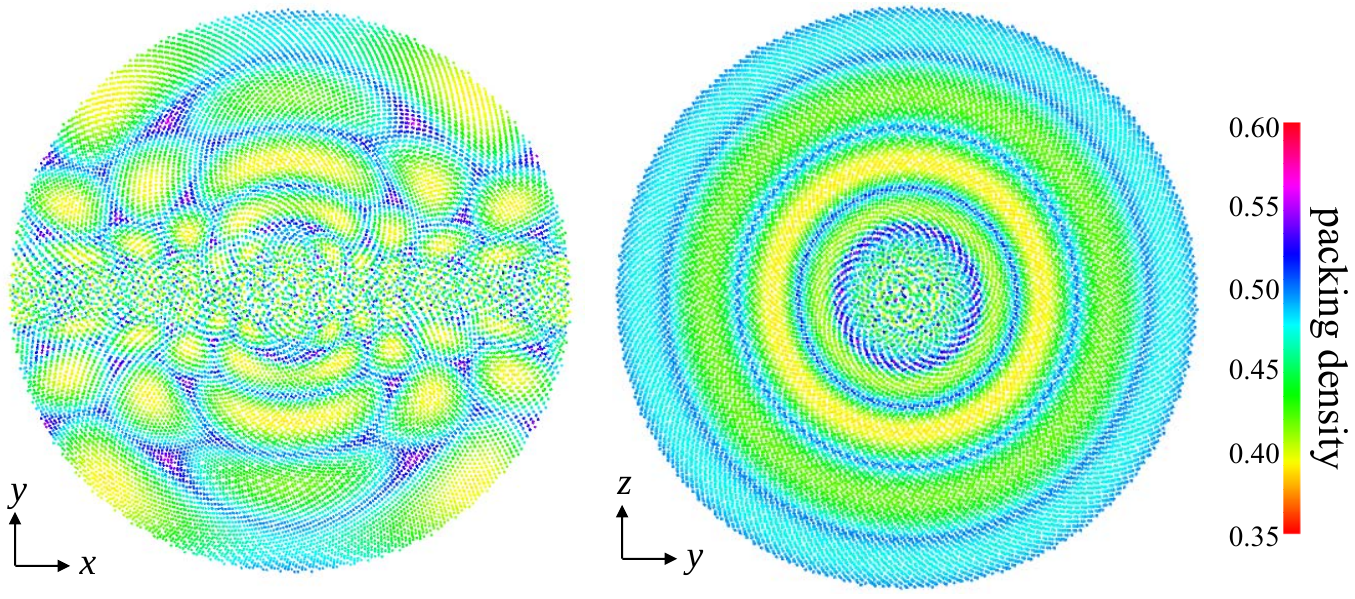}}
\caption{Cross-sections of the 3D Vogel spiral in the $x$-$y$ and $y$-$z$ planes in case of $s = 1$ and $r = R = 1000$.}
\label{fig: 3D Vogel spiral} 
\end{figure}
\end{example}

\section{General case: packing of Riemannian manifolds}
\label{sec: Application to aperiodic packings of the Riemannian manifolds}

The determination of the Riemannian manifolds that can be packed by the proposed method is a problem raised for the first time in this study.
Theorem~\ref{thm: theorem 3} deals with this problem in a local sense in the case of real analytic surfaces (class $C^\omega$). 
The general treatment of the same problem in a global sense is left as a challenging problem for future research.

Let $N \ge n > 0$ be integers, and $f({\mathbf x}) = (f_1({\mathbf x}), \dots, f_n({\mathbf x})) : {\mathfrak D} \rightarrow \RR^N$ be a function defined on an open subset ${\mathfrak D} \subset \RR^{n}$ that fulfills ($\star$), ($\star\star$) in the Introduction section. Thus,
\begin{eqnarray}\label{eq: general form of Jacobian matrix}
D f ({\mathbf x}) = 
U({\mathbf x}) \begin{pmatrix} \Phi({\mathbf x}) \\ O \end{pmatrix}, \quad O: \text{$(N-n) \times n$ zero matrix.}
\end{eqnarray}

As in the previous section, if $n$ and $N$ are fixed, it is possible to derive the PDE systems for the  diagonal entries of $\Phi({\mathbf x})$ and the orthogonal matrix $U({\mathbf x})$.


The Nash embedding theorem states that there is a lower bound $N_0$ only depending on $k$ such 
that any $C^k$ Riemannian $n$-manifolds $(M, g)$ can be isometrically embedded into the Euclidean space $\RR^N$ if $N \ge n_0$,
by an injective map of class $C^k$ ($3 \le k \le \infty$ or $k = \omega$) \cite{Nash1956}, \cite{Nash1966}.

Therefore we can fix an atlas
$\{ (U_\alpha, \varphi_\alpha) \}_{\alpha \in I }$ of $M$ 
and isometries $\iota_\alpha : U_{\alpha } \rightarrow \iota_{\alpha }(U_{\alpha })\subset \RR^{N}$.
Let $D f_\alpha({\mathbf x})$ be the Jacobian matrix of $f_\alpha := \iota_{\alpha} \circ \varphi_{\alpha}^{-1}$.
For another diffeomorphism $\psi_\alpha : U_{\alpha } \to \psi_{\alpha}(U_{\alpha}) \subset \RR^n$ and any open subset $V \subset U_\alpha$,
it is straightforward to see that 
the Jacobian matrix of $\iota_{\alpha} \circ \psi_{\alpha}^{-1}$ fulfills ($\star$) and ($\star\star$) on $V$ if and only if the following ($\flat$) holds:
\begin{itemize}
\item[($\flat$)] The diffeomorphism ${\mathbf x} = \sigma({\mathbf y}) := \varphi_{\alpha} \circ \psi_{\alpha}^{-1}$: $\psi_\alpha(V) \overset{\simeq}{\rightarrow} \varphi_\alpha(V)$ has the Jacobian matrix $D \sigma$ such that 
$\tr{(D \sigma({\mathbf y}))} \tr{(D f_\alpha(\sigma({\mathbf y})))} D f_\alpha(\sigma({\mathbf y})) D \sigma({\mathbf y})$ is diagonal, and 
has a constant determinant.
\end{itemize}

In particular, if 
$\varphi_\alpha$ is an isothermal coordinate system,
\IE 
$\tr{(D f_\alpha({\mathbf x}))} D f_\alpha({\mathbf x}) = \lambda({\mathbf x}) I$ for some positive-valued function $\lambda({\mathbf x})$, 
then ($\flat$) occurs if and only if 
$\tr{(D \sigma({\mathbf y}))} D \sigma({\mathbf y})$ is diagonal,
and
$$\det (\tr{(D \sigma({\mathbf y}))} D \sigma({\mathbf y}) ) = c^{2} \lambda^{-n}( \sigma({\mathbf y})).$$

In Theorem~\ref{thm: theorem 3}, we assume that $n = 2$, and the Riemannian surface $(M, g)$ is real analytic
in order to use the Cauchy-Kovalevskaya theorem.
\begin{theorem}\label{thm: theorem 3}
For any constant $c \ne 0$ and real analytic Riemannian surface $(M, g)$, 
there is an atlas $\{ (U_\alpha, \varphi_\alpha) \}_{\alpha \in I}$ of $M$
such that for any $\alpha \in I$,
we have a real analytic function $\epsilon(x_1, x_2)$ on $\varphi_\alpha(U_\alpha)$ with 
$$
g |_{U_\alpha} = \frac{c^2}{\epsilon(x_1, x_2)} dx_1^2 + \epsilon(x_1, x_2) dx_2^2.
$$
Thus, for any isometry $\iota_\alpha : U_\alpha \rightarrow \iota_\alpha(U_\alpha) \subset \RR_N$,
$\iota_\alpha \circ \varphi_\alpha^{-1}$ fulfills ($\star$) and ($\star\star$).

\end{theorem}
\begin{proof}
For any $p \in M$, fix a neighborhood $p \in U_\alpha \subset M$, a diffeomorphism $\varphi_\alpha : U_\alpha \rightarrow \varphi_\alpha(U_\alpha) \subset \RR^{2}$ and an isometry $\iota_\alpha: U_\alpha \rightarrow\iota_\alpha(U_\alpha) \subset \RR^N$.
We may assume that 
$\varphi_\alpha$ is an isothermal coordinate system of $U_\alpha$.
We shall prove that some neighborhood $V \subset U_\alpha$ of $p$
and a diffeomorphism $\sigma: \sigma^{-1}(\varphi_\alpha(V)) \overset{\simeq}{\rightarrow} \varphi_\alpha(V)$ satisfy ($\flat$)
with respect to $f_\alpha := \iota_\alpha \circ \varphi_\alpha^{-1}$ and $\psi_{\alpha} := \sigma^{-1} \circ \varphi_{\alpha} |_V$.
Hence the chart $(V, \psi_\alpha)$ has the desired property from the above discussion.

$D \sigma^{-1}({\mathbf x}) \tr{(D \sigma^{-1}({\mathbf x}))} = ( \tr{(D \sigma(\sigma^{-1}({\mathbf x})))} D \sigma(\sigma^{-1}({\mathbf x})) )^{-1}$ is diagonal
and has a determinant $c^{-2} \lambda^2({\mathbf x})$ for some positive function $\lambda({\mathbf x})$ on $\varphi_\alpha(V)$
if and only if $D \sigma^{-1}$ 
is represented as follows: 
\begin{eqnarray*}
D \sigma^{-1}(x_1, x_2)
&=& c^{-1} \lambda(x_1, x_2)
\begin{pmatrix}
\epsilon^{1/2}(x_1, x_2) & 0 \\
0 & \pm \epsilon^{-1/2}(x_1, x_2)
\end{pmatrix}
U(x_1, x_2), \\
U(x_1, x_2)
&=& 
\begin{pmatrix}
\cos \theta(x_1, x_2) & \sin \theta(x_1, x_2) \\
- \sin \theta(x_1, x_2) & \cos \theta(x_1, x_2)
\end{pmatrix}.
\end{eqnarray*}

If we put $B_i := (D \sigma^{-1})^{-1} (D \sigma^{-1})_{x_i}$
and $A_i := U^{-1} U_{x_i}$, then
\begin{eqnarray*}
B_i &=& \frac{\lambda_{x_i}}{\lambda} I + \frac{\epsilon_{x_i}}{2 \epsilon} \tr{U} \begin{pmatrix} 1 & 0 \\ 0 & -1 \end{pmatrix} U + A_i \\
&=&
\frac{\lambda_{x_i}}{\lambda} I + \frac{\epsilon_{x_i}}{2 \epsilon} 
\begin{pmatrix}
\cos 2 \theta(x_1, x_2) & \sin 2 \theta(x_1, x_2) \\
\sin 2 \theta(x_1, x_2) & -\cos 2 \theta(x_1, x_2)
\end{pmatrix}
+ \theta_{x_i}
\begin{pmatrix}
0 & 1 \\
-1 & 0
\end{pmatrix}.
\end{eqnarray*}
From $(f_{x_1})_{x_2} = (f_{x_2})_{x_1}$, 
the second column of $(D \sigma^{-1})_{x_1} = (D \sigma^{-1}) B_1$
and the first column of $(D \sigma^{-1})_{x_2} = (D \sigma^{-1}) B_2$
must be equal.
Hence, the second column of $B_1$ and the first column of $B_2$ are also equal, which implies
$$
\frac{\lambda_{x_1}}{\lambda} \begin{pmatrix} 0 \\ 1 \end{pmatrix} + \frac{\epsilon_{x_1}}{2 \epsilon} 
\begin{pmatrix}
\sin 2 \theta(x_1, x_2) \\
-\cos 2 \theta(x_1, x_2)
\end{pmatrix}
+ \theta_{x_1}
\begin{pmatrix}
1 \\
0
\end{pmatrix}
=
\frac{\lambda_{x_2}}{\lambda} \begin{pmatrix} 1 \\ 0 \end{pmatrix} + \frac{\epsilon_{x_2}}{2 \epsilon} 
\begin{pmatrix}
\cos 2 \theta(x_1, x_2) \\
\sin 2\theta(x_1, x_2) 
\end{pmatrix}
+ \theta_{x_2}
\begin{pmatrix}
0 \\
-1
\end{pmatrix}.
$$
Thus,
$$
\begin{pmatrix}
\cos 2 \theta(x_1, x_2) & -\sin 2 \theta(x_1, x_2) \\
\sin 2\theta(x_1, x_2) & \cos 2 \theta(x_1, x_2)
\end{pmatrix}
\begin{pmatrix}
(\log \epsilon)_{x_2} \\
(\log \epsilon)_{x_1}
\end{pmatrix}
= 2
\begin{pmatrix} 
-(\log \lambda)_{x_2} + \theta_{x_1} \\
(\log \lambda)_{x_1} + \theta_{x_2}
\end{pmatrix}.
$$
From $( (\log \epsilon)_{x_2} )_{x_1} = ( (\log \epsilon)_{x_1} )_{x_2}$,
\begin{eqnarray*}
& & \hspace{-5mm} \frac{\partial}{\partial x_1}
\begin{pmatrix}
\cos 2 \theta(x_1, x_2) & \sin 2 \theta(x_1, x_2) \\
\end{pmatrix}
\begin{pmatrix} 
-(\log \lambda)_{x_2} + \theta_{x_1} \\
(\log \lambda)_{x_1} + \theta_{x_2}
\end{pmatrix} \\
&=& \frac{\partial}{\partial x_2}
\begin{pmatrix}
-\sin 2\theta(x_1, x_2) & \cos 2 \theta(x_1, x_2)
\end{pmatrix}
\begin{pmatrix} 
-(\log \lambda)_{x_2} + \theta_{x_1} \\
(\log \lambda)_{x_1} + \theta_{x_2}
\end{pmatrix}.
\end{eqnarray*}
which implies
\begin{eqnarray}\label{eq; PDE for theta}
& & \hspace{-5mm}
\begin{pmatrix}
1 & \tan 2 \theta \\
\end{pmatrix} 
\left(
\begin{pmatrix} 
-(\log \lambda)_{x_1 x_2} + \theta_{x_1 x_1} \\
(\log \lambda)_{x_1 x_1} + \theta_{x_1 x_2}
\end{pmatrix} 
+ 2 \theta_{x_2}
\begin{pmatrix} 
-(\log \lambda)_{x_2} + \theta_{x_1} \\
(\log \lambda)_{x_1} + \theta_{x_2}
\end{pmatrix} 
\right) \\
&=&
\begin{pmatrix}
- \tan 2\theta & 1
\end{pmatrix}
\left( 
\begin{pmatrix} 
-(\log \lambda)_{x_2 x_2} + \theta_{x_1 x_2} \\
(\log \lambda)_{x_1 x_2} + \theta_{x_2 x_2}
\end{pmatrix}
- 2 \theta_{x_1}
\begin{pmatrix} 
-(\log \lambda)_{x_2} + \theta_{x_1} \\
(\log \lambda)_{x_1} + \theta_{x_2}
\end{pmatrix}
\right). \nonumber
\end{eqnarray}

From the Cauchy-Kovalevskaya theorem, Eq.(\ref{eq; PDE for theta}) has a real analytic solution $\theta$ defined on $\varphi_\alpha(V)$ for a simply-connected neighborhood $p \in V \subset U$.
From this $\theta$, $\log \epsilon$ and $\sigma^{-1}(x_1, x_2)$ can be constructed
by solving the corresponding PDEs above.
\qed
\end{proof}

%
%

Next, we prove Theorem~\ref{thm: theorem 4}. 
For any diffeomorphism $\sigma : U \rightarrow V$ between open subsets of two manifolds and metric $g$ on $V$, 
For any $p \in U$ and $u, v \in T_p U$,
the pullback metric $\sigma^* g$ is defined by
$$
(\sigma^* g)_p(u, v) = g_{\sigma(p)}(d \sigma_p(u), d \sigma_p(v)).
$$

\begin{theorem}\label{thm: theorem 4}
For any integers $0 < n < N$, and $2 \le k \le \infty$ or $k = \omega$, let $f({\mathbf x}) \in C^k({\mathfrak D}, \RR^N)$ be a function
on a simply-connected open subset ${\mathfrak D} \subset \RR^n$ with the Jacobian matrix $J_f({\mathbf x})$.
Suppose that $f({\mathbf x})$ satisfies (i) and (ii) for some constant $\alpha \in \RR$ and antisymmetric matrix $A$ of degree $N$: 
\begin{enumerate}[(i)]

\item $\tr{J_f({\mathbf x})} Q({\mathbf x}) J_f({\mathbf x}) = O$ for any ${\mathbf x} \in {\mathfrak D}$, if we put:
\begin{eqnarray*}
f^*({\mathbf x}) &:=& (\alpha I + A) f({\mathbf x}), \\
Q({\mathbf x}) &:=& f^*({\mathbf x}) \tr{f^*({\mathbf x})} A + A f^*({\mathbf x}) \tr{f^*({\mathbf x})} - (f^*({\mathbf x}) \cdot f^*({\mathbf x}) ) A.
\end{eqnarray*}
\item $\det (\tr{J_f}({\mathbf x}) J_f({\mathbf x})) \ne 0$, and $f^*({\mathbf x})$ is linearly independent of $f_{x_1}({\mathbf x}), \ldots, f_{x_n}({\mathbf x})$ over $\RR$ for any ${\mathbf x} \in {\mathfrak D}$.
\end{enumerate}

Let $H_f({\mathbf x})$ be the function on ${\mathfrak D}$ determined by the following equations, except for the constant term:
\begin{eqnarray*}
(H_f)_{x_j}({\mathbf x}) &=& \frac{ f^*({\mathbf x}) \cdot f_{x_j}({\mathbf x}) }{ f^*({\mathbf x}) \cdot f^*({\mathbf x}) } \quad (j = 1, \ldots, n).
\end{eqnarray*}

Let $q_f = (q_{ij})_{1 \le i, j \le n}$ be the positive-definite symmetric matrix.
\begin{eqnarray*}
q_f({\mathbf x}) &:=& e^{ - \frac{2 (n+1)}{n} \alpha H_f({\mathbf x}) } \left( f^*({\mathbf x}) \cdot f^*({\mathbf x}) \right)^{\frac{1}{n}} \tr{J_f}({\mathbf x}) \left( I - \frac{f^*({\mathbf x}) \tr{f^*({\mathbf x})} }{ \tr{f^*({\mathbf x})} f^*({\mathbf x}) } \right) J_f({\mathbf x}).
\end{eqnarray*}
Hence, $g_f = \sum_{i, j=1}^n q_{ij} dx_i dx_j$ is a Riemannian metric on ${\mathfrak D}$.
For any function $h({\mathbf x}, x_{n+1})$ on ${\mathfrak D} \times \RR_{> 0}$,
$F_{f, h}({\mathbf x}, x_{n+1}): {\mathfrak D} \times \RR_{> 0} \rightarrow \RR^N$ is defined by
\begin{eqnarray}\label{eq: f with self similarity 2}
F_{f, h}({\mathbf x}, x_{n+1}) &:=& e^{ \alpha h({\mathbf x}, x_{n+1})} \exp(A h({\mathbf x}, x_{n+1})) f({\mathbf x}).
\end{eqnarray}
The following (a) and (b) are equivalent:
\begin{enumerate}
\item[(a)] $F_{f, h}({\mathbf x}, x_{n+1})$
satisfies ($\star$), ($\star\star$) for some $h({\mathbf x}, x_{n+1}) \in C^k({\mathfrak D} \times \RR_{> 0}, \RR)$. 
\item[(b)] The metric $g_f$ is diagonal and has a constant determinant on ${\mathfrak D}$.
\end{enumerate}

For any $C^k$ diffeomorphism ${\mathbf x} = \sigma({\mathbf y})$ on ${\mathfrak D}$, (a$^\prime$), (b$^\prime$) are also equivalent:
\begin{enumerate}
\item[(a$^\prime$)] $F_{f \circ \sigma, h}({\mathbf y}, x_{n+1})$
fulfills ($\star$), ($\star\star$) for some $h({\mathbf y}, x_{n+1}) \in C^k({\mathfrak D} \times \RR_{> 0}, \RR)$.

\item[(b$^\prime$)] The pull-back $\sigma^* g_f = g_{f \circ \sigma}$ is diagonal and has a constant determinant on ${\mathfrak D}$.
\end{enumerate}

\end{theorem}

\begin{remark}
In (b$^\prime$), such a $\sigma$ exists whenever $n = 1$, or $n = 2$ and $k = \omega$ as a result of Theorem~\ref{thm: theorem 3}.
In (i), $Q({\mathbf x}) = O$ holds whenever $A = O$, or $n = 1$, or $N = 2, 3$ and $\alpha = 0$. 
\end{remark}

\begin{proof}
From (i),
$$
\left( \left( \frac{ f^* \cdot f_{x_j} }{ f^* \cdot f^* } \right)_{x_i} - \left( \frac{ f^* \cdot f_{x_i} }{ f^* \cdot f^* } \right)_{x_j} \right)_{ i, j = 1, \ldots, n}
= \frac{2 \tr{J_f} \left(
f^* \tr{f^*} A + A f^* \tr{f^*} - (f^* \cdot f^*) A
\right) J_f }{ (f^* \cdot f^*)^2 } = O.
$$
The existence of $H_{f}$ is obtained from this and the generalized Stokes theorem for 1-forms of class $C^1$ (Theorem~6.1, XXIII, \cite{Lang93}).
Since $I - J_{f} (\tr{J}_{f} J_{f})^{-1} \tr{J_{f}}$ is the projection onto the linear space generated by $f_{x_1}, \ldots, f_{x_n}$,
the assumption (ii) implies: 
\begin{eqnarray*}
\det q_f({\mathbf x}) 
= e^{ - 2 (n+1) \alpha H_f({\mathbf x}) } ( \tr{ f^*({\mathbf x}) } \left( I - J_f (\tr{J_f} J_f)^{-1} \tr{J_f} \right) f^*({\mathbf x}) ) \det (\tr{J_f}({\mathbf x}) J_f({\mathbf x})) \ne 0.
\end{eqnarray*}
Thus, $q_f({\mathbf x})$ is positive-definite, and $g_f$ is a Riemannian metric on ${\mathfrak D}$.

For any diffeomorphism $\sigma$ on ${\mathfrak D}$,
(i) and (ii) hold for $f$ if and only if they do for $f \circ \sigma$. Thus, 
(a$^\prime$) $\Leftrightarrow$ (b$^\prime$) is immediately obtained from (a) $\Leftrightarrow$ (b).
Hence, only (a) $\Leftrightarrow$ (b) is proved in the following.

To show (a) $\Rightarrow$ (b), we assume that $F_{f, h}({\mathbf x}, x_{n+1})$ satisfies ($\star$) and ($\star\star$). 
\begin{eqnarray*}
(F_{f, h})_{x_j}({\mathbf x}, x_{n+1}) & = & e^{\alpha h} \exp(A h) ( h_{x_j} f^*({\mathbf x}) + f_{x_j}({\mathbf x}) ), \quad (j=1, \ldots, n) \\
(F_{f, h})_{x_{n+1}}({\mathbf x}, x_{n+1}) & = & e^{\alpha h} \exp(A h) h_{x_{n+1}} f^*({\mathbf x}).
\end{eqnarray*}

Because of ($\star\star$),
$h_{x_{n+1}}({\mathbf x}, x_{n+1}) \ne 0$ must hold for any $({\mathbf x}, x_{n+1}) \in {\mathfrak D} \times \RR_{> 0}$.
In addition, 
$(F_{f, h})_{x_j} \cdot (F_{f, h})_{x_{n+1}} = 0$ ($j = 1, \ldots, n$) from ($\star$), and $f^*({\mathbf x}) \cdot f^*({\mathbf x}) \ne 0$ from (ii).
These imply:
\begin{eqnarray*}
h_{x_j}({\mathbf x}, x_{n+1}) 
= - \frac{ f^*({\mathbf x}) \cdot f_{x_j}({\mathbf x}) }{ f^*({\mathbf x}) \cdot f^*({\mathbf x}) }.
\end{eqnarray*}

Therefore, $h({\mathbf x}, x_{n+1})$ fulfills for some function $h_0(x_{n+1})$:
\begin{eqnarray}\label{eq: 1st eq for h}
h({\mathbf x}, x_{n+1}) = -H_f({\mathbf x}) + h_0(x_{n+1}).
\end{eqnarray}

We also have: 
\begin{eqnarray*}
(F_{f, h})_{x_j}({\mathbf x}, x_{n+1}) & = & e^{\alpha h} \exp(A h) \left( I - \frac{f^*({\mathbf x}) \tr{f^*({\mathbf x})} }{ \tr{f^*({\mathbf x})} f^*({\mathbf x}) } \right) f_{x_j}({\mathbf x}) \quad (j=1, \ldots, n).
\end{eqnarray*}

Thus, using the Jacobian matrix $J_{f}({\mathbf x})$ of $f({\mathbf x})$,
\begin{eqnarray}\label{eq: n times n part}
((F_{f, h})_{x_i} \cdot (F_{f, h})_{x_j})_{1 \le i, j \le n} &=& e^{2 \alpha h} \tr{J_f}({\mathbf x}) \left( I - \frac{f^*({\mathbf x}) \tr{f^*({\mathbf x})} }{ \tr{f^*({\mathbf x})} f^*({\mathbf x}) } \right) J_f({\mathbf x}) \\
&=& e^{2 \alpha (h + \frac{n+1}{n} H_f) } \left( f^*({\mathbf x}) \cdot f^*({\mathbf x}) \right)^{ - \frac{ 1 }{ n } }
q_f({\mathbf x}). \nonumber
\end{eqnarray}

$F({\mathbf x}, x_{n+1})$ fulfills ($\star\star$)
if and only if the following holds for some constant $c\ne 0$:
\begin{eqnarray*}\label{eq: starstar for F}
\prod_{j=1}^{n+1} ((F_{f, h})_{x_j} \cdot (F_{f, h})_{x_j}) 
&=& h_{ x_{n+1} }^2 e^{2 (n+1)\alpha (h + H_f) } \det q_f({\mathbf x}) \\
&=& ( h_0^\prime(x_{n+1}) )^2 e^{2 (n+1)\alpha h_0(x_{n+1}) } \det q_f({\mathbf x})
= c^2,
\end{eqnarray*}
which implies:
%
$\det q_f({\mathbf x}) = \gamma^2$ for some $\gamma \ne 0$.
Thus, (a) $\Rightarrow$ (b) is proved.

Conversely, the above discussion shows that (b) $\Rightarrow$ (a) is obtained if 
there is a constant $d_1 \ne 0$ such that
$h_0^\prime(x_{n+1}) e^{(n+1)\alpha h_0(x_{n+1}) } = d_1$.
Namely,
\begin{eqnarray}\label{eq: final eq for h}
h({\mathbf x}, x_{n+1})
&=&
\begin{cases}
-H_f({\mathbf x}) + \displaystyle\frac{1}{(n+1) \alpha} \log (d_1 x_{n+1} + d_2)
& \text{if } \alpha \ne 0, \\
-H_f({\mathbf x}) + d_1 x_{n+1} + d_2
& \text{if } \alpha = 0.
\end{cases}
\end{eqnarray}
In particular, if $d_1 = 1$ and $d_2 = 0$, the above $h({\mathbf x}, x_{n+1})$ is defined on ${\mathfrak D} \times \RR_{>0}$, which proves (b) $\Rightarrow$ (a). \qed
\end{proof}

As seen from Eq.(\ref{eq: n times n part}), 
if $N = n + 1$ and a diffeomorphism $\sigma$ on ${\mathfrak D}$ satisfies (b$^\prime$) of Theorem~\ref{thm: theorem 4} for some $3 \le k \le \infty$ or $k = \omega$,
the following $\phi_{j}^2({\mathbf x}, x_{n+1})$ ($j=1, \ldots, n$)
and 
$\phi_{n+1}^2({\mathbf x}, x_{n+1}) = c^2/ \prod_{j=1}^n \phi_{j}^2({\mathbf x}, x_{n+1})$
are a solution of the PDEs of Theorem~\ref{prop: PDE system}.
\begin{eqnarray*}
\phi_{j}^2({\mathbf x}, x_{n+1}) &=& e^{2 \alpha h({\mathbf x}, x_{n+1})} \left( (f (\sigma({\mathbf x})) )_{x_j} \cdot (f ( \sigma({\mathbf x})) )_{x_j} - \frac{ \left( (f (\sigma({\mathbf x})) )_{x_j} \cdot f^*(\sigma({\mathbf x})) \right)^2 }{ f^*(\sigma({\mathbf x})) \cdot f^*(\sigma({\mathbf x})) } \right), \\
\end{eqnarray*}
where $h({\mathbf x}, x_{n+1})$ is the function obtained by replacing $H_f$ in Eq.(\ref{eq: final eq for h}) with $H_{f \circ \sigma}$. 

The 3D Vogel spiral can be obtained from Theorem~\ref{thm: theorem 4} and the following parameters as the case of $n = 2$, $N= 3$: 
\begin{eqnarray*}
\alpha = 0, \quad
A =
\begin{pmatrix}
0 & 0 & 0 \\
0 & 0 & 1 \\
0 &-1 & 0
\end{pmatrix}, \quad 
f(s, t) = s^{1/3}
\begin{pmatrix}
t \\
0 \\
\sqrt{r^2 - t^2}
\end{pmatrix}.
\end{eqnarray*}

\section*{Acknowledgments}
  The authors would like to thank Mr.~Y.~Azama and Ms.~C.~Ooisi of Kyushu University for their help in coding.
  This project was financially supported by the SENTAN-Q program of MEXT Initiative for Realizing Diversity in the Research Environment, JST-Mirai Program (JPMJMI18GD) and JSPS KAKENHI (19K03628).
  The first author participated as an RA of the program, and was also employed by the project of JST CREST (JPMJCR1911). 
  The research by the second author was supported by grants from the Research Grants Council of the Hong Kong SAR, China (HKU 17301317 and 17303618).

\section*{Author contributions}
The last author designed the project, and performed the majority of the mathematical research.
The first author derived the PDEs and coded programs with the last author.
The second author supervised the project as the mentor of the SENTAN-Q program.

\bibliography{refs}{}   
\bibliographystyle{plain}      

\end{document}